\documentclass[11pt, oneside, reqno]{amsart}
\usepackage[letterpaper, margin=1in]{geometry}
\usepackage{parskip}
\usepackage{eucal}
\usepackage{amsmath,amssymb,amsfonts,amsthm,mathrsfs}
\usepackage{tikz}
\usepackage{graphicx, array}
\usepackage{enumerate}
\usepackage{epsfig}
\usepackage{float}
\usepackage{color}
\usepackage{xcolor}
\colorlet{BLUE}{blue}
\usepackage{enumitem}
\usepackage[colorlinks=true, urlcolor=blue, linkcolor=red]{hyperref}

\newtheorem{theorem}{Theorem}[section]
\newtheorem{lemma}{Lemma}[section]

\newtheorem{proposition}{Proposition}[section]
\theoremstyle{definition}
\newtheorem{definition}{Definition}[section]

\newtheorem{remark}{Remark}[section]
\numberwithin{equation}{section}

\newcommand{\vet}[4]{\begin{bmatrix}#1 \\ #2 \\#3\\#4\end{bmatrix}}

\newcommand{\sech}{\mathrm{sech}}
\def\longdelete#1{}

\begin{document}

\author[Hsiao, T.-Y.]{Ting-Yang Hsiao}
\address{International School for Advanced Studies (SISSA),
Via Bonomea 265, 34136 Trieste, Italy}
\email{thsiao@sissa.it}

\author[Liang, Z.]{Zhengjun Liang}
\address{Department of Mathematics, University of Michigan, 530 Church St, Ann Arbor, MI 48104}
\email{jspliang@umich.edu}

\author[To, G.]{Giang To}
\address{Lund University, Sweden}
\email{giang.to@math.lu.se}

\author[Zhang, Y.]{Ye Zhang}
\address{International School for Advanced Studies (SISSA),
Via Bonomea 265, 34136 Trieste, Italy}
\email{yezhang@sissa.it}

%\thanks{Corresponding authors: Ting-Yang Hsiao and Ye Zhang.}

\title[Existence of pure capillary solitary waves in constant vorticity flows]
{Existence of pure capillary solitary waves in constant vorticity flows}

\date{\today}

% \date{\today}

%We prove the existence of pure capillary solitary waves in $2D$ finite-depth flows with nonzero constant vorticity. This resolves the remaining open case in the pure-capillary regime, complementing the recent nonexistence result of Ifrim–Pineau–Tataru–Taylor for irrotational flows. Our analysis is based on a spatial–dynamics formulation of the travelling-wave Euler equations, followed by a nonlinear coordinate transformation that simultaneously straightens the free boundary and rectifies the symplectic structure. This yields an infinite-dimensional reversible Hamiltonian system whose linearization possesses a two-dimensional center subspace along a distinguished curve in the vorticity–capillary  parameter space. Using a parameter-dependent version of Mielke’s center-manifold theorem, we obtain a finite-dimensional canonical Hamiltonian system governing the near-critical dynamics. After expanding the reduced Hamiltonian to cubic order and performing a long-wave rescaling, we derive a KdV-type normal form admitting a reversible homoclinic orbit. This orbit persists in the full water-wave problem and produces the desired solitary waves. Our result provides the first rigorous construction of pure capillary solitary waves with constant vorticity in finite depth, establishing a sharp contrast between the rotational and irrotational settings and clarifying the dispersive mechanism responsible for the bifurcation.

\begin{abstract}
We prove that the finite-depth pure-capillary rigidity mechanism in the
irrotational water-wave problem is destroyed by a suitable constant-vorticity
critical shear. More precisely, we construct small-amplitude finite-depth pure
capillary solitary waves for the two-dimensional free-boundary Euler equations
with nonzero constant vorticity and zero gravity. The waves bifurcate from a
critical shear flow whose relative horizontal velocity vanishes at the bed, so
that the standard Dubreil--Jacotin no-stagnation formulation is singular at the
asymptotic state. We therefore formulate the traveling-wave problem directly as
a Hamiltonian spatial-dynamics system in flattened Euler variables, remove a
nonlinear boundary condition from the domain of the vector field, and verify the
spectral and resolvent hypotheses needed for a two-dimensional center-manifold
reduction. A parameter-dependent Darboux transformation and a cubic expansion
of the reduced Hamiltonian yield, under a long-wave scaling, a stationary KdV
equation. Its reversible homoclinic orbit persists under the full reduced
dynamics and gives a family of small-amplitude waves of depression.
\end{abstract}

% Might need to shorten the abstract

\maketitle

%\setcounter{tocdepth}{1}
%\tableofcontents

\section{Introduction}

Pure-capillary solitary waves in finite depth exhibit a striking rigidity in
the irrotational setting: localized traveling waves do not exist. A natural
question is whether this obstruction is caused by the absence of gravity itself,
or rather by the irrotational no-shear geometry. In this paper we show that the
latter is the case. We construct finite-depth pure-capillary solitary waves
with nonzero constant vorticity, bifurcating from a critical shear flow whose
relative horizontal velocity vanishes at the bed.

We study solitary waves for the two-dimensional
finite-depth water-wave problem with surface tension and nonzero constant
vorticity, in the absence of gravity. Let the fluid occupy a time-dependent
domain $\mathcal D(t)\subset\mathbb R^2$, with flat bottom $\{y=-d\}$ and free
surface $\Gamma(t)$. We assume unit density and nonzero constant vorticity
$\omega$. The velocity field $\mathbf u$ and pressure $p$ solve the
incompressible Euler equations
\begin{subequations}\label{eq:water-wave-system}
\begin{equation}\label{eq:euler-freebdry}
\begin{aligned}
&\mathbf u_t+\mathbf u\cdot\nabla \mathbf u=-\nabla p
        && \text{in }\mathcal D(t),\\
&\nabla\cdot\mathbf u=0,\
  \operatorname{curl}\mathbf u=\omega
        && \text{in }\mathcal D(t).
\end{aligned}
\end{equation}
The boundary conditions are
\begin{alignat}{2}
\mathcal V(\Gamma(t)) &= \mathbf u\cdot \mathbf n
        &\quad &\text{on }\Gamma(t), \label{kinematic_top}\\
0 &= \mathbf u\cdot \mathbf e_2
        &\quad &\text{on }\{y=-d\}, \label{kinematic_bottom}\\
p &= \sigma H
        &\quad &\text{on }\Gamma(t), \label{Young-Laplace}
\end{alignat}
\end{subequations}
where $\mathbf n$ is the outward unit normal to the free surface,
$\mathcal V(\Gamma(t))$ is the normal velocity of the moving boundary, $H$ is
the signed curvature of $\Gamma(t)$, and $\sigma>0$ is the surface tension
coefficient.

A \emph{traveling wave} solution of \eqref{eq:water-wave-system} is a solution
whose shape is stationary in a frame moving with constant horizontal speed
$c$. More precisely, for some $c\neq0$, the fluid domain, free surface,
velocity, and pressure satisfy
\begin{equation}
\mathcal D(t)=\mathcal D(0)+(ct,0),\qquad
\Gamma(t)=\Gamma(0)+(ct,0),
\end{equation}
and
\begin{equation}
\mathbf u(t,x,y)=\mathbf U(x-ct,y),\qquad
p(t,x,y)=P(x-ct,y).
\end{equation}
Equivalently, after introducing the moving coordinate $\xi=x-ct$, the solution
is independent of time in the $(\xi,y)$-variables. A traveling wave is called
a \emph{solitary wave} if the wave profile is localized and decays to an
equilibrium state at spatial infinity.

\subsection{Main result and interpretations}

Our main result constructs a family of small-amplitude finite-depth
pure-capillary solitary waves with nonzero constant vorticity. We consider
waves whose free surface is a graph
\begin{equation}
\Gamma(t)=\{(x,y):y=\eta(t,x)\},
\end{equation}
and, with slight abuse of notation, we also denote the traveling-wave profile
by $\eta(\xi)$.

\begin{theorem}[Existence of finite-depth pure-capillary waves]\label{thm:main}
Let
\begin{equation}
\frac{\sigma}{c^2d}>\frac13 .
\end{equation}
There exists $\varepsilon_0>0$ such that for every
$0<\varepsilon<\varepsilon_0$ satisfying
\begin{equation}
\frac{\omega d}{c}=1+\varepsilon ,
\end{equation}
the system \eqref{eq:water-wave-system} admits a $C^2$ solitary traveling
wave with speed $c$. Moreover, the solitary-wave profile has the asymptotic
expansion
\begin{equation}
\eta(\xi)=d\, \varepsilon\, Q\bigg(\varepsilon^{\frac{1}{2}}\bigg(\frac{\sigma}{c^2d}-\frac13\bigg)^{-\frac{1}{2}}\frac{\xi}{d}\bigg)+O(\varepsilon^2)
\end{equation}
where the remainder is uniform in the nondimensional traveling coordinate and
\begin{equation}
Q(X)=-3\sech^2\left(\frac{X}{2}\right)
\end{equation}
solves the stationary KdV equation
\begin{equation}
Q''=Q+\frac12 Q^2 .
\end{equation}
In particular, the leading-order wave is a wave of depression; see
Figure~\ref{fig:wave_profile}.
\end{theorem}

The pure-capillary problem is a limiting regime in which surface tension is
the only restoring mechanism at the free surface. This regime is natural at
small Bond number,
\begin{equation}
\mathrm{Bo}:=\frac{\rho g L^2}{\sigma}\ll1,
\end{equation}
and also in reduced-gravity or density-matched settings. Thus the assumption
$g=0$ should not be understood merely as a formal deletion of gravity from
ordinary macroscopic water waves. Rather, it isolates the part of the
free-boundary Euler dynamics governed by capillarity, vorticity, and finite
depth.

In the irrotational finite-depth pure-capillary problem, solitary waves are
known not to exist \cite{ifrim2022no}. Thus the irrotational nonexistence
theorem leaves open a sharp structural question: is the obstruction caused by
the absence of gravity, or by the irrotational no-shear geometry? Theorem~\ref{thm:main}
shows that the obstruction is not stable under the introduction of critical
shear. More precisely, pure capillarity by itself does not rule out localized
traveling waves once the flow is allowed to have a suitable shear. The
mechanism is geometric. Surface tension provides the restoring force, finite
depth provides the long-wave dispersion, and constant vorticity creates the
critical shear structure from which the solitary wave bifurcates.

\begin{figure}[!htbp]
\centering
\begin{tikzpicture}[scale=1.2]
% water fill below the profile
\fill[teal!10, domain=-4:4, samples=200]
plot (\x, {-0.9*(1/cosh(\x))^2}) -- (4,-3.15) -- (-4,-3.15) -- cycle;
\draw[line width=1.4pt, teal, domain=-4:4, samples=200, smooth]
    plot (\x, {-0.9*(1/cosh(\x))^2});
\end{tikzpicture}
\caption{Leading-order wave profile}
\label{fig:wave_profile}
\end{figure}

\subsection{Discussions of the result and relationship with previous work}

We first explain why the parameter regime in Theorem~\ref{thm:main} is
natural, and why the corresponding nonlinear problem is not a straightforward
variant of the classical gravity--capillary theory.

If gravity is temporarily restored and $c=d=1$, then infinitesimal periodic
waves with wave number $k$ satisfy the dispersion relation
\begin{equation}\label{eq:intro-dispersion}
k\coth k=(g+\omega)+\sigma k^2 .
\end{equation}
Thus the pure-capillary rotational choice
\begin{equation}
g=0,\qquad \omega=1+\varepsilon
\end{equation}
has the same long-wave linear balance as the irrotational gravity--capillary
choice
\begin{equation}
g=1+\varepsilon,\qquad \omega=0 .
\end{equation}
For $\sigma>1/3$, the value $g+\omega=1$ gives a double root at $k=0$, and the
perturbation $g+\omega=1+\varepsilon$ produces the long-wave KdV scaling.

This formal coincidence is geometrically misleading. In the irrotational
gravity--capillary problem, the far-field velocity is zero in the laboratory
frame, and hence the relative horizontal velocity in the moving frame is
\begin{equation}
c-u_{1,\infty}=1 .
\end{equation}
In the pure-capillary rotational problem, the same coefficient in
\eqref{eq:intro-dispersion} is produced by the shear. In the moving frame, the
far-field velocity is
\begin{equation}
u_{1,\infty}(y)-c=-(\omega y+1),
\qquad -1\le y\le0,
\end{equation}
and therefore
\begin{equation}\label{eq:far_field_shear_velocity}
c-u_{1,\infty}(y)=1+\omega y .
\end{equation}
At the critical value $\omega=1$, this relative velocity vanishes at the bed.
For the branch in Theorem~\ref{thm:main}, where $\omega=1+\varepsilon$, the
zero occurs at
\begin{equation}
y_\varepsilon=-\frac{1}{1+\varepsilon}\in(-1,0).
\end{equation}
Thus the asymptotic state already contains a critical level. The waves
constructed in this paper should therefore be viewed as bifurcating from a
critical shear flow, not from a no-stagnation laminar flow.

This is the central geometric obstruction. The standard Dubreil--Jacotin
height formulation used in the no-stagnation theory of rotational water waves
requires the relative horizontal velocity $c-u_1$ to have a fixed sign. If
$\Psi$ is the relative stream function,
\begin{equation}
\Psi_y=u_1-c,\qquad \Psi_x=-u_2,
\end{equation}
and $s=-\Psi(x,y)$, then
\begin{equation}
\partial_y s=c-u_1 .
\end{equation}
Hence the use of $s$ as a vertical coordinate is legitimate only under a
no-stagnation condition. Equivalently, for the inverse height function
$y=h(x,s)$, one has
\begin{equation}
h_s=\frac{1}{c-u_1}.
\end{equation}
In the regime considered here, \eqref{eq:far_field_shear_velocity} shows that
this denominator vanishes already in the far-field shear. Consequently, the
desired branch lies outside the height-function phase space underlying the
spatial-dynamics construction of Groves--Wahl\'en \cite{groves2007spatial}
and related no-stagnation approaches.

We therefore formulate the traveling-wave problem directly in flattened Euler
variables. This is not only a technical preference but a necessity imposed by
the critical-level geometry. In this formulation one must keep the elliptic
bulk variable, the bottom boundary, the shear, and the capillary boundary
condition inside a single infinite-dimensional Hamiltonian system. The
Hamiltonian functional, the weak symplectic form, the reversible structure,
the linearized operator, and the resolvent estimates must all be derived in
this Eulerian phase space.

A further difficulty is that the natural Hamiltonian vector field has a
nonlinear boundary condition in its domain. Before applying center-manifold
theory, we perform a nonlinear change of variables which fixes the domain of
the vector field while preserving the reversible Hamiltonian structure. After
this preparation, the linearized operator at the critical shear has a
two-dimensional center subspace and a hyperbolic complement satisfying the
required resolvent bounds. The dynamics near the critical shear is therefore
reduced to a two-dimensional Hamiltonian center manifold.

The reduced symplectic form is still parameter-dependent and not initially in
canonical form. We use a parameter-dependent Darboux transformation to obtain
canonical coordinates on the center manifold. We then compute the reduced
Hamiltonian to cubic order. Under the long-wave scaling dictated by the
double root of \eqref{eq:intro-dispersion}, the reduced Hamiltonian system
converges to the stationary KdV equation
\begin{equation}
Q''=Q+\frac12 Q^2 .
\end{equation}
Its reversible homoclinic orbit persists for the full reduced dynamics. The
final step is not merely a projection argument: the homoclinic orbit on the
center manifold is lifted through the fixed-domain transformation and then
through the Eulerian reconstruction theorem, producing a solution of the
original free-boundary Euler equations. This yields the solitary waves in
Theorem~\ref{thm:main}.

Let us also compare the present finite-depth problem with the infinite-depth
pure-capillary construction of Rowan--Wan \cite{rowan2024two}. In infinite
depth, they use conformal mapping and a boundary integral formulation to
derive a Babenko-type equation for the free surface, and the bifurcation near
the critical velocity is governed by a stationary focusing cubic nonlinear
Schr\"odinger equation. The finite-depth problem considered here has a
different low-frequency structure. In the infinite-depth Babenko formulation,
the critical linear symbol has the form
\begin{equation}\label{eq:infinite_depth_symbol}
\ell_*(\xi)
=
c_*\omega-c_*^2|\xi|+\sigma|\xi|^2 .
\end{equation}
In finite depth, the corresponding low-frequency symbol is instead
\begin{equation}\label{eq:finite_depth_symbol}
\widetilde\ell_*(\xi)
=
c_*\omega-c_*^2|\xi|\coth(|\xi|d)+\sigma|\xi|^2 .
\end{equation}
The bottom boundary therefore changes the small-frequency behavior through
the factor $\coth(|\xi|d)$. In particular, the inverse estimates and the
normal-form coefficients are different from their infinite-depth analogues.
Together with the critical-level obstruction described above, this makes the
finite-depth problem a genuinely different Hamiltonian spatial-dynamics
problem rather than a direct adaptation of the infinite-depth scalar
reduction.

The main contribution of the paper is thus twofold. First, we prove the first
finite-depth existence result for pure-capillary solitary waves with nonzero
constant vorticity, showing that the irrotational pure-capillary
nonexistence mechanism is not stable under the introduction of critical
shear. Second, we develop a Hamiltonian center-manifold construction for a
finite-depth rotational water-wave problem outside the standard
Dubreil--Jacotin no-stagnation framework. This identifies the geometric
mechanism by which constant vorticity creates solitary waves in a regime
where pure capillarity alone is rigid.

\subsection{Literature review}

Traveling water waves are among the central special solutions of the
free-boundary Euler equations. The classical theory begins with Russell's
observation of solitary waves~\cite{russell1844report} and the asymptotic
theories of Boussinesq and Korteweg--de Vries. Rigorous existence theory for
two-dimensional solitary waves was first developed in the irrotational
finite-depth gravity setting by Friedrichs--Hyers~\cite{friedrichs1954existence},
Beale~\cite{thomas1977existence}, and Amick--Toland~\cite{amick1981periodic}.
For finite-depth gravity--capillary waves, the spatial-dynamics and
center-manifold approach was initiated by Kirchg\"assner~\cite{kirchgassner1988nonlinearly}
and Amick--Kirchg\"assner~\cite{amick1989theory}, and was further developed in
many works, including
\cite{buffoni1996plethora,iooss1992water,buffoni1999multiplicity,
constantin2011nonlinear,groves2004steady,groves2002dimension,
groves2008fully,puaruau2005nonlinear,haragus2002finite,deng2009three}.
Infinite-depth solitary waves in the gravity--capillary regime were
constructed in
\cite{buffoni2004existence,buffoni2004existence2,groves2011existence,
iooss1996capillary}.

The limiting pure-gravity and pure-capillary regimes exhibit stronger
rigidity. In the irrotational setting, Hur~\cite{hur2012no} proved
nonexistence of two-dimensional pure-gravity solitary waves in deep water
under a decay assumption, while Ifrim--Tataru~\cite{ifrim2020no} proved
nonexistence in deep water for both the pure-gravity and pure-capillary
problems. Most directly related to the present work,
Ifrim--Pineau--Tataru--Taylor~\cite{ifrim2022no} proved that
two-dimensional finite-depth irrotational pure-capillary solitary waves do
not exist. Theorem~\ref{thm:main} gives the complementary finite-depth
rotational existence result in the constant-vorticity setting, and shows that
the obstruction in the irrotational pure-capillary theory can be overcome by
critical shear.

Rotational water waves with surface tension have also been extensively
studied in the periodic setting. Wahl\'en~\cite{wahlen2006steady} constructed
steady periodic capillary--gravity waves with vorticity. For constant
vorticity, Constantin--Varvaruca~\cite{constantin2011steady} studied
regularity and local bifurcation for periodic gravity waves, while Martin
proved regularity and local bifurcation results for periodic capillary waves
\cite{martin2012regularity,martin2013local} and for the corresponding
capillary--gravity problem~\cite{martin2013capillarygravity}. Martin--Matioc
constructed Wilton ripples with constant vorticity and capillary effects
\cite{martinMatioc2013wilton}. Related Stokes expansions, numerical studies,
and recent variational bifurcation results can be found in
\cite{hsu2016gravity,barbieri2025bifurcation}. These works provide a broad
periodic-wave background for the interaction between vorticity and surface
tension.

For solitary waves with vorticity, Groves--Wahl\'en~\cite{groves2007spatial}
constructed small-amplitude gravity--capillary solitary waves with arbitrary
vorticity distribution by spatial-dynamics methods, and later proved
existence and conditional energetic stability for constant-vorticity
gravity--capillary solitary waves by variational methods
\cite{grovesWahlen2015existence}. Critical layers and stagnation phenomena
are another important feature of rotational water waves; see, for example,
\cite{constantin2011steady,kozlov2020solitary}. In the pure-capillary
rotational direction, Hur--Wheeler~\cite{hur2020exact} constructed exact
free-surface solutions with constant vorticity, Rowan--Wan~\cite{rowan2024two}
constructed infinite-depth pure-capillary solitary waves with nonzero constant
vorticity, and Kharif--Abid--Chen--Hsu~\cite{kharif2025nonlinear} derived an
NLS approximation for pure capillary waves on vertically sheared currents.
The present paper fills the finite-depth constant-vorticity gap in this
pure-capillary solitary-wave theory by constructing localized waves through a
Hamiltonian center-manifold analysis at a critical shear flow.

\subsection{Organization of the paper}

The rest of the paper is organized as follows.
\begin{enumerate}
\item In Section~\ref{sec:formulation}, we derive the traveling-wave
formulation and construct the Hamiltonian spatial-dynamics system in
flattened Euler variables. We also prove that solutions of this Eulerian
Hamiltonian system reconstruct solutions of the original free-boundary
Euler equations.

\item In Section~\ref{sec:center}, we remove the nonlinear boundary
condition from the domain of the Hamiltonian vector field by a nonlinear
transformation. We then formulate the resulting reversible Hamiltonian
system on a fixed phase space and state the Hamiltonian center-manifold and
Darboux reduction used later.

\item In Section~\ref{sec:normalform}, we verify the spectral and
resolvent hypotheses, identify the two-dimensional center space, compute
the cubic Hamiltonian coefficients, and derive the KdV homoclinic profile.
Finally, we lift the homoclinic orbit back to the physical variables and
prove the asymptotic formula in Theorem~\ref{thm:main}. Section~\ref{sec:App1}
collects auxiliary coefficient computations and estimates used in the
normal-form expansion.
\end{enumerate}

\subsection*{Acknowledgements}
T.-Y. Hsiao is supported by the European Union ERC Consolidator Grant 2023
GUnDHam, Project Number 101124921. He would like to express his sincere
gratitude to Vera Hur, Zhao Yang, and Alberto Maspero for their invaluable
guidance and continuous encouragement. He also thanks Chongchun Zeng and
Erik Wahl\'en for their insightful advice and generous help during the
completion of this project.

Z. Liang was supported in part by NSF grant DMS-2153992 through his advisor
during the Winter 2025 semester, when part of this work was completed. He
would like to thank his doctoral advisor Sijue Wu for her continuous support, and Zach Deiman, Noah Stevenson, and Yuchuan Yang for helpful discussions.

G. To was supported by the Swedish Research Council, grant no. 2020-00440.

Y. Zhang is supported by the European Union ERC Starting Grant 2020 GeoSub,
Project Number 945655.

\section{Notation and conventions}

Throughout the paper, manifolds are understood to be \(C^k\) Hilbert
manifolds, for the relevant value of \(k\). This convention includes
finite-dimensional \(C^k\) manifolds as a special case. We refer to
\cite{Abraham1988} for background. Points of a manifold \(M\) are denoted by
\(m\in M\). When \(M\) is an open subset of a Hilbert space \(H\), we identify
\(T_mM\) with \(H\). For a vector field or differential form \(V\) on \(M\), we
write \(V(m)\) or \(V_m\) for its value at \(m\).

Many maps depend on auxiliary parameters. We write such maps as \(f(x,\lambda)\),
where \(x\) is the manifold variable and \(\lambda\) is a parameter. For fixed
\(\lambda\), we also write
\[
        f^\lambda(x):=f(x,\lambda).
\]
Unless otherwise stated, operations such as composition, inversion, and
differentiation are performed with respect to the manifold variable, with the
parameters held fixed.

If \(L:H_1\times\cdots\times H_n\to H'\) is an \(n\)-linear map, we write
\[
        L[v_1,\ldots,v_n]
\]
for its value on \((v_1,\ldots,v_n)\). For a map \(f:M\to H'\), its
Fr\'echet differential at \(m\) is denoted by
\[
        df(m):T_mM\to H',
\]
and we write \(df(m)[v]\), \(df_m[v]\), or simply \(df(m)v\) for its action on
\(v\in T_mM\). Higher Fr\'echet derivatives are denoted by \(d^n f(m)\), viewed
as \(n\)-linear maps:
\[
        d^n f(m)[v_1,\ldots,v_n].
\]

If \(f\) is defined on a product manifold \(M_1\times\cdots\times M_n\), then
\(d_i f\) denotes the differential with respect to the \(i\)-th component. In
Euclidean coordinates, we use the standard notation
\[
        f_x,\quad f_y,\quad f_{xx},\quad f_{xy},\quad f_{yy},\quad \ldots
\]
for partial derivatives. When the variable is one-dimensional, for instance
\(y\in(a,b)\), we may also write \(\dot f\) or \(f_y\) for \(df/dy\). The same
notation is used for weak derivatives whenever the meaning is clear from
context.

\section{Hamiltonian formulation in flattened Euler variables}\label{sec:formulation}
The goal of this section is to derive a Hamiltonian dynamical system equivalent to the solitary wave problem corresponding to \eqref{eq:water-wave-system}. Because the critical shear lies outside the no-stagnation height-function framework, the formulation is carried out directly in flattened Euler variables. 
\subsection{Zakharov--Craig--Sulem formulation}
The water-wave system \eqref{eq:water-wave-system} is often studied in the
Zakharov--Craig--Sulem framework, which we briefly recall below. For a more detailed derivation, see, for example, \cite{Wahlen2007HamiltonianConstantVorticity}.  

We assume that our free boundary is a graph $\partial\mathcal D(t)=\{(x,y): y=\eta(t,x)\}$. Since the vorticity $\mathrm{curl}\,\mathbf u = \omega$ is a constant, the velocity field $\mathbf u$ can be written as a shear flow plus an irrotational perturbation as
\begin{equation}\label{velocity}
\mathbf u = (-\omega y,0)+\nabla\phi.
\end{equation}
Then, in terms of the velocity potential $\phi$, its harmonic conjugate $\psi$, and the free boundary function $\eta$, \eqref{eq:euler-freebdry}--\eqref{Young-Laplace} can be recast as
\begin{equation}\label{main second}
\begin{aligned}
&\phi_{xx}+\phi_{yy}=0 && \text{in } -d<y<\eta(t,x),\\
&\phi_y=0 && \text{on } y=-d,\\
&\eta_t=\phi_y-(\phi_x-\omega \eta)\eta_x && \text{on } y=\eta(t,x),\\
&\phi_t+\frac12(\phi_x^2+\phi_y^2)-\omega\eta\,\phi_x+\omega\psi
-\frac{\sigma\,\eta_{xx}}{(1+\eta_x^2)^{3/2}}=0 && \text{on } y=\eta(t,x).
\end{aligned}
\end{equation}
Since we study the traveling wave problem, we use the moving-frame variable $\xi:= x - ct$ to obtain
\begin{equation}\label{main third}
\begin{aligned}
&\phi_{\xi\xi}+\phi_{yy}=0 && \text{in } -d<y<\eta(\xi),\\
&\phi_y=0 && \text{on } y=-d,\\
&-c\,\eta_{\xi}=\phi_y-(\phi_{\xi}-\omega \eta)\eta_{\xi} && \text{on } y=\eta(\xi),\\
&-c\,\phi_{\xi}+\frac12(\phi_{\xi}^2+\phi_y^2)-\omega\eta\,\phi_{\xi}+\omega\psi
-\frac{\sigma\,\eta_{\xi\xi}}{(1+\eta_{\xi}^2)^{3/2}}=0 && \text{on } y=\eta(\xi).
\end{aligned}
\end{equation}
Note that the third equation in \eqref{main third} implies
\begin{equation}\label{psi-eta=0 1}
\frac{d}{d\xi}\Bigl(\psi(\xi,\eta(\xi))-\frac{\omega}{2}\eta(\xi)^2-c\eta(\xi)\Bigr)=0.
\end{equation}
and thus
\begin{equation}\label{psi-eta=0 2}
\psi(\xi,\eta(\xi))=\frac{\omega}{2}\eta(\xi)^2+\eta(\xi)+\mathrm{const}.
\end{equation}
For solitary waves we require $\eta(\xi)\to 0$ as $|\xi|\to\infty$, so the constant must be zero. Furthermore, under the rescalings
\[
\xi\mapsto \frac{\xi}{d},\quad y\mapsto \frac{y}{d},\quad \omega\mapsto \frac{\omega d}{c},\quad
\sigma\mapsto \frac{\sigma}{c^2 d},\quad\phi\mapsto \frac{\phi}{cd},\quad \psi\mapsto \frac{\psi}{cd},\quad \eta\mapsto \frac{\eta}{d},
\]
we may normalize $c=d=1$ in \eqref{main third}. Substituting \eqref{psi-eta=0 2} into \eqref{main third} yields the elliptic system
\begin{equation}\label{main fourth}
\begin{aligned}
&\phi_{\xi\xi}+\phi_{yy}=0 && \text{in } -1<y<\eta(\xi),\\
&\phi_y=0 && \text{on } y=-1,\\
&-\eta_{\xi}=\phi_y-(\phi_{\xi}-\omega \eta)\eta_{\xi} && \text{on } y=\eta(\xi),\\
&-\phi_{\xi}+\frac12\phi_{\xi}^2+\frac12\phi_y^2-\omega\eta\,\phi_{\xi}
+\omega\eta+\frac{\omega^2}{2}\eta^2
-\frac{\sigma\,\eta_{\xi\xi}}{(1+\eta_{\xi}^2)^{3/2}}=0 && \text{on } y=\eta(\xi).
\end{aligned}
\end{equation}

For notational simplicity, in the following we still write $x$ in place of $\xi$.

\subsection{The Eulerian Hamiltonian phase space}
Our next step is to formulate \eqref{main fourth} as an infinite-dimensional quasi-linear Hamiltonian system. For $s = 1, 2$, we define the Hilbert space 
\begin{align} \label{X^s}
X^s=H^s(0,1)\times H^s(0,1)\times\mathbb{R}\times\mathbb{R}, 
\end{align}
whose norm is given by
\begin{align}\label{normXs}
\|(\varphi,\theta,z,\eta)\|_{X^s} := \|\varphi\|_{H^s(0,1)} + \|\theta\|_{H^s(0,1)} + |z| + |\eta|.
\end{align}

\begin{remark}
For $\theta\in H^1(0,1)$, we define the traces by
\begin{equation}\label{eq:trace_identities}
\theta(0) := - \int_0^1 (1 - y) \theta_y \; dy + \int_0^1 \theta \; dy,\quad 
\theta(1) := \int_0^1 y \theta_y \; dy + \int_0^1 \theta \; dy.
\end{equation}
These functionals are well-defined, and the Cauchy--Schwarz inequality gives
\begin{align}\label{traceb}
    |\theta(0)| + |\theta(1)| \le C \|\theta\|_{H^1(0,1)}.
\end{align}
Furthermore, we have the integration by parts formula
\begin{align}\label{IBPH1}
    \varphi(1)\theta(1) - \varphi(0)\theta(0)
    = \int_0^1 \varphi_y \theta + \varphi \theta_y \; dy,
\end{align}
for $H^1(0,1)$ functions. The formula (\ref{IBPH1}) holds first for $C^1$ functions, and we observe that both sides of it are bounded with respect to $H^1(0,1)$ norm, so a standard limiting argument passes it to $H^1(0,1)$ functions. 
\end{remark}

We then consider the Hilbert space  
\begin{align} \label{M_0}
    M_0=\left\{ (\varphi,\theta,z,\eta)\in X^1: \int_0^1 \varphi \;dy=0, \theta(0)=\theta(1)=0 \right\}
\end{align}
and the open manifold $M^{\omega, \sigma}:= \left\{ m\in M_0: |\tilde{z}|<\sigma,~ \eta > -1 \right\}$ where 
\begin{align} \label{z bar}
    \tilde{z}:=z +\int_{0}^1 \dfrac{y{\varphi}_{y}\left({\theta}_{y}+\omega (1-y)-1\right)}{\eta+1}\; dy.
\end{align} 
Note that $M^{\omega, \sigma}$ is an open subset of $M_0$. To define the Hamiltonian system as an ODE on the phase space $M^{\omega, \sigma}$, we first recall the following definition of a symplectic form.
\begin{definition}[Symplectic form {\cite[Supplement~6.4A]{Abraham1988}}]
Let \(M\) be a manifold. A smooth \(2\)-form \(\Omega\) on \(M\) is called a
\emph{symplectic form} if the following two conditions hold:
\begin{enumerate}[label=(\roman*)]
    \item \(\Omega\) is closed, namely $d\Omega=0$.
    \item \(\Omega\) is weakly nondegenerate at every point \(m\in M\): if \(v\in T_mM\) satisfies $\Omega_m[v,w]=0$ for all $w\in T_mM$, then \(v=0\).
\end{enumerate}
\end{definition}
 
Consider the symplectic $2$-form on $M^{\omega, \sigma}$ defined by
\begin{equation}\label{symplectic form}
  \Omega_m^{\omega, \sigma}[v^1,v^2]:=z^2\eta^1-\eta^2 z^1+ \int_{0}^1 \left(\theta^2_y\varphi^1-\varphi^2\theta^1_y\right)\;dy,
\end{equation}
where $v^1=(\varphi^1,\theta^1,z^1,\eta^1)$ and $v^2=(\varphi^2,\theta^2,z^2,\eta^2)$ are on $M_0\cong T_m M^{\omega, \sigma}$. $\Omega^{\omega, \sigma}$ is closed because it does not depend on $m$ and the weak nondegeneracy is straightforward to check. We then define a Hamiltonian functional $H^{\omega, \sigma}$ such that 
\begin{equation}\label{second Hamiltonian}
\begin{aligned}
      &H^{\omega, \sigma}(m)=\int_{0}^1  \dfrac{1}{2(\eta+1)}\left(\left({\theta}_{y} +\omega (1-y)+\eta\right)^2-\varphi^2_{y}\right) dy\\
    &\qquad+\int_{0}^1 \omega (y(\eta+1)-1)\left({\theta}_y -1 + \omega (1-y)\right)\;dy-\sqrt{\sigma^2-\tilde{z}^2} - \frac{1}{2} \omega + \frac{1}{6} \omega^2 + \sigma.
\end{aligned}
\end{equation}
Note that $H^{\omega, \sigma}$ is smooth. The triple $(M^{\omega, \sigma},\Omega^{\omega, \sigma},H^{\omega, \sigma})$ forms a Hamiltonian system. To be more precise, given $m\in M^{\omega, \sigma} \cap X^2$ with 
\begin{align} \label{BC first}
    \varphi_y(0)=0,\,\,\,\mbox{and}\,\,\,\varphi_y(1)-\frac{\tilde{z}(\theta_y(1)-1)}{\sqrt{\sigma^2-\tilde{z}^2}}=0,
\end{align}
we can determine a unique Hamiltonian vector field $v_H^{\omega, \sigma}$ such that for every tangent vector $v'_m\in M_0 \cong T_m M^{\omega, \sigma}$, 
\begin{align} \label{def of u_H}
     d H_m^{\omega, \sigma}[v'_m]=\Omega_m^{\omega, \sigma}[(v_H^{\omega, \sigma})_m,v'_m].
\end{align}
A straightforward calculation reveals that if we write $v_H^{\omega, \sigma} = (\varphi_H^{\omega, \sigma}, \theta_H^{\omega, \sigma},z_H^{\omega, \sigma}, \eta_H^{\omega, \sigma})$, then
\begin{equation}\label{eq:final_Hamiltonian_eqns}
\begin{split}
{\varphi}_{H}^{\omega, \sigma}(m)
&=
\dfrac{1}{\eta+1}
\left(
{\theta}_y
+\dfrac{\tilde{z}(y{\varphi}_y-\varphi(1))}
{\sqrt{\sigma^2-\tilde{z}^2}}
\right)
+\omega(\eta+1)\left(y-\dfrac{1}{2}\right)
-\dfrac{\omega}{\eta+1}\left(y-\dfrac{1}{2}\right),
\\
{\theta}_{H}^{\omega, \sigma}(m)
&=
\dfrac{1}{\eta+1}
\left(
\dfrac{\tilde{z}y({\theta}_y+\omega(1-y)-1)}
{\sqrt{\sigma^2-\tilde{z}^2}}
-{\varphi}_y
\right),
\\
z_{H}^{\omega, \sigma}(m)
&=
\int_0^1
\left(
\dfrac{1}{2(\eta+1)^2}
\Big(({\theta}_y+\omega(1-y)-1)^2-\varphi_y^2\Big)
-\dfrac{1}{2}
\right)\,dy
\\
&\quad
+\dfrac{\tilde{z}}{\sqrt{\sigma^2-\tilde{z}^2}}
\int_0^1
\dfrac{y{\varphi}_y({\theta}_y+\omega(1-y)-1)}
{(\eta+1)^2}\,dy
-\int_0^1 \omega y({\theta}_y+\omega(1-y)-1)\,dy,
\\
\eta_{H}^{\omega, \sigma}(m)
&=
\dfrac{\tilde{z}}{\sqrt{\sigma^2-\tilde{z}^2}}.
\end{split}
\end{equation}
Henceforth we use $N_0$ to denote the space $M_0 \cap X^2$ equipped with the norm $\|\cdot\|_{X^2}$, and we use $N^{\omega, \sigma}$ to denote the open manifold $M^{\omega, \sigma} \cap X^2$ of $N_0$ equipped with $\|\cdot\|_{X^2}$. From \eqref{eq:final_Hamiltonian_eqns} we see that $v_H^{\omega, \sigma}$ can be extended to a smooth function from $N^{\omega, \sigma}$ to $X^1$. We denote
\begin{align*}
    \mathrm{dom}(v_H^{\omega, \sigma})=\{(\varphi,\theta,z,\eta)\in N^{\omega, \sigma}: \eqref{BC first}\,\,\mbox{holds}\}.
\end{align*}
Note that if  $m = (\varphi,\theta,z,\eta) \in  \mathrm{dom}(v_H^{\omega, \sigma})$, then $v_H^{\omega, \sigma}(m) \in M_0$. The Hamiltonian system is then given by
\begin{align} \label{old Hamiltons eq}
    \dot{\gamma}(x)=v_H^{\omega, \sigma}(\gamma(x)),
\end{align}
which is, after writing $\gamma(x)=(\varphi(x,\cdot),\theta(x,\cdot),z(x),\eta(x))$,
\begin{equation}\label{Old Hamilton's equations 2}
\begin{aligned}
\dot{\varphi}&=\dfrac{1}{\eta+1}\left({\theta}_y+\dfrac{\tilde{z}(y{\varphi}_y-\varphi(1))}{\sqrt{\sigma^2-\tilde{z}^2}}\right)+\omega(\eta+1)\left(y-\dfrac{1}{2}\right)-\dfrac{\omega}{\eta+1}\left(y-\dfrac{1}{2}\right),\\  
\dot{\theta}&=\dfrac{1}{\eta+1}\left( \dfrac{\tilde{z}y({\theta}_y +\omega(1-y)-1)}{\sqrt{\sigma^2- \tilde{z}^2}}-{\varphi}_y \right),\\
\dot{z}&= \int_0^1 \left(\dfrac{1}{2(\eta+1)^2}\left(({\theta}_y +\omega(1-y)-1)^2-\varphi^2_y\right)-\dfrac{1}{2}\right) dy\\ 
&\,\,+\dfrac{\tilde{z}}{\sqrt{\sigma^2-\tilde{z}^2}}\int_0^1 \dfrac{y{\varphi}_y({\theta}_y +\omega(1-y)-1)}{(\eta+1)^2}dy -\int_0^1 \omega y({\theta}_y +\omega (1-y)-1) dy,
\\ 
\dot{\eta}&= \dfrac{\tilde{z}}{\sqrt{\sigma^2- \tilde{z}^2}}.
\end{aligned}
\end{equation}
We end this subsection with some symmetric properties of $H^{\omega, \sigma}$, $\Omega^{\omega, \sigma}$, and $v_H^{\omega, \sigma}$. 

\begin{proposition}[Reversibility]\label{prop:reversibility}
Let \(S:M_0\to M_0\) be the linear isometry defined by
\[
        S(\varphi,\theta,z,\eta)=(-\varphi,\theta,-z,\eta).
\]
Then \(S^2=\mathrm{id}\), \(S(M^{\omega,\sigma})=M^{\omega,\sigma}\), and
\(S\) is an anti-symplectic symmetry of the Hamiltonian system. More precisely,
for every \(m\in M^{\omega,\sigma}\) and \(v^1,v^2\in T_mM^{\omega,\sigma}\cong M_0\),
\begin{align}
        H^{\omega,\sigma}(Sm)
        &=
        H^{\omega,\sigma}(m),
        \label{eq:symH}\\
        \Omega_{Sm}^{\omega,\sigma}[Sv^1,Sv^2]
        &=
        -\Omega_m^{\omega,\sigma}[v^1,v^2].
        \label{eq:symOmega}
\end{align}
Consequently, the Hamiltonian vector field \(v_H^{\omega,\sigma}\) is reversible
with respect to \(S\):
\begin{equation}\label{eq:symvH}
        v_H^{\omega,\sigma}(Sm)
        =
        -S v_H^{\omega,\sigma}(m),
        \quad m\in M^{\omega,\sigma}.
\end{equation}
\end{proposition}

\begin{proof}
The identities \(S^2=\mathrm{id}\) and
\(S(M^{\omega,\sigma})=M^{\omega,\sigma}\) follow directly from the definitions:
under \(S\), both \(\varphi\) and \(z\) change sign, while \(\theta\) and \(\eta\)
are unchanged; in particular, the quantity \(\tilde z\) changes sign, so that
the conditions \(|\tilde z|<\sigma\) and \(\eta>-1\) are preserved.

The identities \eqref{eq:symH} and \eqref{eq:symOmega} are obtained by direct
substitution in the definitions of \(H^{\omega,\sigma}\) and
\(\Omega^{\omega,\sigma}\).

It remains to prove \eqref{eq:symvH}. Recall that the Hamiltonian vector field satisfies
\begin{equation}
\Omega_m^{\omega,\sigma}
        [v_H^{\omega,\sigma}(m),v]
        =
        dH^{\omega,\sigma}(m)[v],
        \quad v\in T_mM^{\omega,\sigma}.
\end{equation}
Using \(H^{\omega,\sigma}\circ S=H^{\omega,\sigma}\), we have
\begin{equation}
dH^{\omega,\sigma}(Sm)[Sv]
        =
        dH^{\omega,\sigma}(m)[v].
\end{equation}
On the other hand, by \eqref{eq:symOmega},
\begin{equation}
\Omega_{Sm}^{\omega,\sigma}
        [-Sv_H^{\omega,\sigma}(m),Sv]
        =
        \Omega_m^{\omega,\sigma}
        [v_H^{\omega,\sigma}(m),v]
        =
        dH^{\omega,\sigma}(m)[v].
\end{equation}
Therefore
\begin{equation}
\Omega_{Sm}^{\omega,\sigma}
        [-Sv_H^{\omega,\sigma}(m),Sv]
        =
        dH^{\omega,\sigma}(Sm)[Sv]
\end{equation}
for every \(v\in T_mM^{\omega,\sigma}\). Since \(S\) maps
\(T_mM^{\omega,\sigma}\) onto \(T_{Sm}M^{\omega,\sigma}\), weak
nondegeneracy of \(\Omega^{\omega,\sigma}\) gives
\begin{equation}
v_H^{\omega,\sigma}(Sm)
        =
        -Sv_H^{\omega,\sigma}(m)
\end{equation}
as desired.
\end{proof}

\subsection{Reconstruction of the free-boundary Euler problem}
We now show that the flattened Eulerian Hamiltonian system is equivalent to the traveling-wave free-boundary problem. This reconstruction is the step which returns the spatial-dynamics solution to the original Euler variables.
\begin{theorem}[Hamiltonian reduction]\label{thm:Hamiltonian_equivalence}
Suppose that \(\gamma\in C^1((a,b),N^{\omega,\sigma})\) satisfies
\eqref{old Hamiltons eq} and
\(\gamma((a,b))\subset\operatorname{dom}(v_H^{\omega,\sigma})
\subset M^{\omega,\sigma}\). Write $\gamma(x) = (\varphi(x,\cdot),\theta(x,\cdot),z(x),\eta(x))$. Let \(a<x_1<x_2<b\), and set $\Theta := \{(x,y)\in\mathbb R^2: x_1<x<x_2,\ -1<y<\eta(x)\}.$
For \((x,y)\in\Theta\), define
\begin{equation}\label{eq:flattening-ybar}
        \bar y
        :=
        \frac{y+1}{\eta(x)+1},
        \qquad
        y=(\eta(x)+1)\bar y-1.
\end{equation}
Using the representatives from Lemma~\ref{lemalmost}, define
\(\phi,\zeta\in H^2(\Theta)\) by
\begin{equation}\label{eq:reconstruction-phi}
\begin{aligned}
\phi(x,y)
&=
\varphi(x,\bar y)
+\frac{2-\omega}{2}(x-x_1)
+\int_{x_1}^x
\left(
\frac{\eta_x(s)}{\eta(s)+1}\varphi(s,1)
+\frac{\omega}{2}\eta(s)
+\frac{\omega-2}{2(\eta(s)+1)}
\right)\,ds,
\end{aligned}
\end{equation}
and
\begin{equation}\label{eq:reconstruction-zeta}
        \zeta(x,y)
        =
        \theta(x,\bar y)
        -\frac{\omega}{2}(1-\bar y)^2
        +\bar y\,\eta(x).
\end{equation}
Then \(\phi\) and \(\zeta\) satisfy
\begin{equation}\label{eq:reconstructed-CR}
        \phi_x=\zeta_y+\omega y,
        \qquad
        \phi_y=-\zeta_x
        \qquad
        \text{in }\Theta.
\end{equation}
Moreover, the pair \((\phi,\eta)\) satisfies the system
\eqref{main fourth} on \(x_1<x<x_2\).
\end{theorem}

Before proving the reconstruction theorem, we record a representative lemma
which allows us to pass from the Hilbert-space formulation to functions on the
physical strip. This step is needed because the Hamiltonian system is posed as
an evolution equation in the spatial variable with values in a Sobolev phase
space. The lemma provides measurable representatives with enough two-variable
Sobolev regularity to justify the reconstruction identities below.

\begin{lemma}\label{lemalmost}
Suppose \(\gamma \in C^1((a,b),N^{\omega,\sigma})\) satisfies
\eqref{old Hamiltons eq} and
\(\gamma((a,b))\subset \mathrm{dom}(v_H^{\omega,\sigma})\subset M^{\omega,\sigma}\).
Write
\[
        \gamma(x)=(\varphi(x,\cdot),\theta(x,\cdot),z(x),\eta(x)).
\]
Then there exist representatives \(\tilde{\varphi}\) and \(\tilde{\theta}\) which are
measurable on \((a,b)\times(0,1)\) and belong to
\(H^2((x_1,x_2)\times(0,1))\) for every \(x_1,x_2\) with
\(a<x_1<x_2<b\). These functions satisfy
\begin{align}\label{equivarphi}
\tilde{\varphi}_x
&=\dfrac{1}{\eta(x)+1}
\left(
\tilde{\theta}_y
+\dfrac{\tilde{z}_*(x)(y\tilde{\varphi}_y-\tilde{\varphi}(x,1))}
{\sqrt{\sigma^2-\tilde{z}_*^2(x)}}
\right)
+\omega(\eta(x)+1)\left(y-\dfrac{1}{2}\right)
-\dfrac{\omega}{\eta(x)+1}\left(y-\dfrac{1}{2}\right),\\
\label{equitheta}
\tilde{\theta}_x
&=\dfrac{1}{\eta(x)+1}
\left(
\dfrac{\tilde{z}_*(x)y(\tilde{\theta}_y+\omega(1-y)-1)}
{\sqrt{\sigma^2-\tilde{z}_*^2(x)}}
-\tilde{\varphi}_y
\right),
\end{align}
where
\begin{align} \label{equizs}
    \tilde{z}_*(x)
    := z(x)+\int_0^1
    \dfrac{y\tilde{\varphi}_{y}
    \left(\tilde{\theta}_{y}+\omega(1-y)-1\right)}
    {\eta(x)+1}\,dy
    =\tilde{z}(x),
    \qquad |\tilde{z}_*|<\sigma,
\end{align}
in the weak sense on \((x_1,x_2)\times(0,1)\), with boundary conditions
\begin{align}\label{eqb1}
\tilde{\theta}(x,0) &= 0, \\
\label{eqb2}
\tilde{\theta}(x,1) &= 0, \\
\label{eqb3}
\tilde{\varphi}_y(x,0) &= 0, \\
\label{eqb4}
\tilde{\varphi}_y(x,1)
&= \frac{\tilde{z}_*(x)(\tilde{\theta}_y(x,1)-1)}
{\sqrt{\sigma^2-\tilde{z}_*^2(x)}}.
\end{align}
Moreover, \(\tilde{\varphi}\) has zero mean in the \(y\)-variable:
\begin{align}\label{equivar}
    \int_0^1 \tilde{\varphi}(x,y)\,dy=0.
\end{align}
Finally, the functions \(z,\eta\in C^2((a,b),\mathbb R)\) satisfy
\begin{align}\label{equiz}
z_x(x)
&= \int_0^1
\left(
\dfrac{1}{2(\eta(x)+1)^2}
\left((\tilde{\theta}_y+\omega(1-y)-1)^2-\tilde{\varphi}_y^2\right)
-\dfrac{1}{2}
\right)\,dy \\
\notag
&\quad
+\dfrac{\tilde{z}_*(x)}{\sqrt{\sigma^2-\tilde{z}_*^2(x)}}
\int_0^1
\dfrac{y\tilde{\varphi}_y(\tilde{\theta}_y+\omega(1-y)-1)}
{(\eta(x)+1)^2}\,dy
-\int_0^1 \omega y(\tilde{\theta}_y+\omega(1-y)-1)\,dy,
\\
\label{equieta}
\eta_x(x)
&= \dfrac{\tilde{z}_*(x)}{\sqrt{\sigma^2-\tilde{z}_*^2(x)}},
\end{align}
again in the weak sense.
\end{lemma}

\begin{proof}
It suffices to construct the representatives on an arbitrary fixed interval
\((x_1,x_2)\subset(a,b)\). The construction below is compatible under restriction:
if \((x_1,x_2)\subset(x_1',x_2')\subset(a,b)\), then the representatives constructed
on \((x_1',x_2')\times(0,1)\) agree with those constructed on
\((x_1,x_2)\times(0,1)\), up to a set of measure zero. A countable extension
therefore gives measurable representatives on \((a,b)\times(0,1)\).

Fix \((x_1,x_2)\subset(a,b)\). Since
\(\gamma\in C^1((a,b),N^{\omega,\sigma})\), we have
\(\varphi,\theta\in C^1((a,b),H^2(0,1))\), and hence
\(\varphi,\theta\in H^1((x_1,x_2),H^1(0,1))\). By
\cite[Part II, Lemma~10.1]{Friedman1969}, there exist
\(\tilde{\varphi},\tilde{\theta}\in H^1((x_1,x_2)\times(0,1))\) such that, for
almost every \(x\in(x_1,x_2)\),
\[
        \tilde{\varphi}(x,\cdot)=\varphi(x,\cdot),
        \qquad
        \tilde{\theta}(x,\cdot)=\theta(x,\cdot).
\]
Since \(\varphi,\theta\in C((a,b),H^2(0,1))\), we also have
\(\varphi_{yy},\theta_{yy}\in L^2((x_1,x_2),L^2(0,1))\). By the same argument,
there exist \(\bar{\varphi},\bar{\theta}\in L^2((x_1,x_2)\times(0,1))\) such
that, for almost every \(x\in(x_1,x_2)\),
\[
        \bar{\varphi}(x,\cdot)=\varphi_{yy}(x,\cdot),
        \qquad
        \bar{\theta}(x,\cdot)=\theta_{yy}(x,\cdot).
\]
For every smooth function \(\psi\) with compact support in
\((x_1,x_2)\times(0,1)\), we have
\begin{equation}
\begin{aligned}
   \int_{x_1}^{x_2}\int_0^1
   \tilde{\varphi}(x,y)\psi_{yy}(x,y)\,dxdy &= \int_{x_1}^{x_2}
   \left(
   \int_0^1 {\varphi}(x,y)\psi_{yy}(x,y)\,dy
   \right)\,dx \\
   &= \int_{x_1}^{x_2}
   \left(
   \int_0^1 {\varphi}_{yy}(x,y)\psi(x,y)\,dy
   \right)\,dx \\
   &=
   \int_{x_1}^{x_2}\int_0^1
   \bar{\varphi}(x,y)\psi(x,y)\,dxdy.
\end{aligned}
\end{equation}
Thus \(\tilde{\varphi}_{yy}=\bar{\varphi}\in L^2((x_1,x_2)\times(0,1))\), and
for almost every \(x\in(x_1,x_2)\),
\begin{equation}
\tilde{\varphi}_{yy}(x,\cdot)
        =
        \bar{\varphi}(x,\cdot)
        =
        \varphi_{yy}(x,\cdot).
\end{equation}
The same argument gives $\tilde{\theta}_{yy}=\bar{\theta}\in L^2((x_1,x_2)\times(0,1))$ and $\tilde{\theta}_{yy}(x,\cdot)=\theta_{yy}(x,\cdot)$ for almost every \(x\in(x_1,x_2)\). The identities
\eqref{equizs}--\eqref{equivar}, \eqref{equiz}, and \eqref{equieta} then follow
directly from these representatives and the corresponding properties of
\eqref{old Hamiltons eq}.

We now prove \eqref{equivarphi} and \eqref{equitheta}. By
\cite[Chapitre~IV, \S~3, Th\'eor\`eme~III]{Schwartz1966}, it is enough to test
against functions of the form \(\phi(x)\psi(y)\), where
\(\phi\in C_c^\infty(x_1,x_2)\) and \(\psi\in C_c^\infty(0,1)\). Fix such
\(\phi\) and \(\psi\). From the second equation of \eqref{old Hamiltons eq},
the map
\[
        x\mapsto \int_0^1 \theta(x,y)\psi(y)\,dy
\]
is \(C^1\), and for almost every \(x\in(x_1,x_2)\),
\begin{equation}
\begin{aligned}
\frac{d}{dx}
\left(
\int_0^1 \theta(x,y)\psi(y)\,dy
\right)
&=
\int_0^1 \dot{\theta}(x,y)\psi(y)\,dy \\
&=
\int_0^1
\dfrac{1}{\eta(x)+1}
\left(
\dfrac{\tilde{z}_*(x)y(\tilde{\theta}_y+\omega(1-y)-1)}
{\sqrt{\sigma^2-\tilde{z}_*^2(x)}}
-\tilde{\varphi}_y
\right)
\psi(y)\,dy .
\end{aligned}
\end{equation}
Consequently,
\begin{equation}
\begin{aligned}
 &\int_{x_1}^{x_2}\int_0^1
 \tilde{\theta}(x,y)\phi_x(x)\psi(y)\,dxdy \\
 &=
 \int_{x_1}^{x_2}
 \left(
 \int_0^1 \theta(x,y)\psi(y)\,dy
 \right)
 \phi_x(x)\,dx \\
 &=
 -\int_{x_1}^{x_2}
 \frac{d}{dx}
 \left(
 \int_0^1 \theta(x,y)\psi(y)\,dy
 \right)
 \phi(x)\,dx \\
 &=
 -\int_{x_1}^{x_2}\int_0^1
 \dfrac{1}{\eta(x)+1}
 \left(
 \dfrac{\tilde{z}_*(x)y(\tilde{\theta}_y+\omega(1-y)-1)}
 {\sqrt{\sigma^2-\tilde{z}_*^2(x)}}
 -\tilde{\varphi}_y
 \right)
 \phi(x)\psi(y)\,dxdy.
\end{aligned}
\end{equation}
This proves \eqref{equitheta}. The proof of \eqref{equivarphi} is identical.

It is also immediate from \eqref{old Hamiltons eq} that
\(z,\eta\in C^2((a,b),\mathbb R)\). It remains to prove
\(\tilde{\varphi},\tilde{\theta}\in H^2((x_1,x_2)\times(0,1))\). Since
\(\tilde{\varphi}_{yy}\) and \(\tilde{\theta}_{yy}\) exist and belong to
\(L^2((x_1,x_2)\times(0,1))\), taking weak derivatives with respect to \(y\) in
\eqref{equivarphi} and \eqref{equitheta} shows that
\(\tilde{\varphi}_{xy}\) and \(\tilde{\theta}_{xy}\) also exist and belong to
\(L^2((x_1,x_2)\times(0,1))\). The same argument gives
\(\tilde{\varphi}_{xx}\) and \(\tilde{\theta}_{xx}\), once we know that the weak
derivative of
\[
        \Phi(x):=\tilde{\varphi}(x,1)
\]
exists on \((x_1,x_2)\). To see this, let \(\phi\in C_c^\infty(x_1,x_2)\) and
choose \(\psi_\varepsilon\in C_c^\infty(0,1)\) with
\(\psi_\varepsilon\to1\) as \(\varepsilon\to0^+\). By trace identities \eqref{eq:trace_identities},
\begin{equation}
\begin{aligned}
    &\int_{x_1}^{x_2} \Phi(x)\phi_x(x)\,dx \\
    &=
    \int_{x_1}^{x_2}\int_0^1
    y\tilde{\varphi}_y(x,y)\phi_x(x)\,dxdy
    +
    \int_{x_1}^{x_2}\int_0^1
    \tilde{\varphi}(x,y)\phi_x(x)\,dxdy \\
    &=
    \lim_{\varepsilon\to0^+}
    \left(
    \int_{x_1}^{x_2}\int_0^1
    y\tilde{\varphi}_y(x,y)\phi_x(x)\psi_\varepsilon(y)\,dxdy
    +
    \int_{x_1}^{x_2}\int_0^1
    \tilde{\varphi}(x,y)\phi_x(x)\psi_\varepsilon(y)\,dxdy
    \right) \\
    &=
    -\lim_{\varepsilon\to0^+}
    \left(
    \int_{x_1}^{x_2}\int_0^1
    y\tilde{\varphi}_{xy}(x,y)\phi(x)\psi_\varepsilon(y)\,dxdy
    +
    \int_{x_1}^{x_2}\int_0^1
    \tilde{\varphi}_x(x,y)\phi(x)\psi_\varepsilon(y)\,dxdy
    \right) \\
    &=
    -\int_{x_1}^{x_2}\int_0^1
    y\tilde{\varphi}_{xy}(x,y)\phi(x)\,dxdy
    -
    \int_{x_1}^{x_2}\int_0^1
    \tilde{\varphi}_x(x,y)\phi(x)\,dxdy \\
    &=
    -\int_{x_1}^{x_2}
    \tilde{\varphi}_x(x,1)\phi(x)\,dx.
\end{aligned}
\end{equation}
Hence \(\Phi_x=\tilde{\varphi}_x(\cdot,1)\). The proof is complete.
\end{proof}

\begin{proof}[Proof of Theorem~\ref{thm:Hamiltonian_equivalence}]
We write \(\tilde z\) for the function \(\tilde z_*\) in \eqref{equizs}. All
identities below are understood weakly on \(\Theta\). Differentiating the
reconstruction formulas \eqref{eq:reconstruction-phi} and
\eqref{eq:reconstruction-zeta}, with \(y=(1+\eta(x))\bar y-1\), gives
\begin{align}
(1+\eta)\phi_y
&= \varphi_{\bar y}, \label{eq:rec-phi-y}\\
(1+\eta)\zeta_y
&= \theta_{\bar y}+\omega(1-\bar y)+\eta, \label{eq:rec-zeta-y}\\
\phi_x+\bar y\eta_x\phi_y
&= \varphi_x+\frac{\eta_x}{1+\eta}\varphi(x,1)
+1+\frac{\omega}{2}(\eta-1)+\frac{\omega-2}{2(1+\eta)}, \label{eq:rec-phi-x}\\
\zeta_x+\bar y\eta_x\zeta_y
&= \theta_x+\bar y\eta_x. \label{eq:rec-zeta-x}
\end{align}
Using \eqref{equitheta} and \eqref{equieta}, we have
\begin{equation}\label{eq:theta-x-substitution}
\theta_x(x,\bar y)
=
\frac{1}{1+\eta}
\left(
\eta_x\bar y\bigl(\theta_{\bar y}(x,\bar y)+\omega(1-\bar y)-1\bigr)
-\varphi_{\bar y}(x,\bar y)
\right).
\end{equation}
Combining \eqref{eq:rec-phi-y}, \eqref{eq:rec-zeta-y},
\eqref{eq:rec-zeta-x}, and \eqref{eq:theta-x-substitution} yields
\begin{equation}\label{eq:zeta-x-minus-phi-y}
        \zeta_x=-\phi_y.
\end{equation}
Similarly, combining \eqref{eq:rec-phi-y}--\eqref{eq:rec-zeta-x} with
\eqref{equivarphi} gives
\begin{equation}\label{eq:zeta-y-phi-x}
        \zeta_y=\phi_x-\omega y.
\end{equation}
It follows that
\begin{alignat}{2}
        \phi_{xx}+\phi_{yy}&=0
        &\quad &\text{in }\Theta, \label{eq:phi-harmonic}\\
        \zeta_{xx}+\zeta_{yy}&=-\omega
        &\quad &\text{in }\Theta. \label{eq:zeta-poisson}
\end{alignat}
Thus \(\phi\) and \(\zeta+\frac{\omega}{2}y^2\) are harmonic, and hence smooth
in the interior.

From \eqref{eq:reconstruction-zeta} and \eqref{eqb1}--\eqref{eqb2}, the stream
function satisfies
\begin{equation}\label{eq:zeta-boundary-values}
        \zeta(x,-1)=-\frac{\omega}{2},
        \qquad
        \zeta(x,\eta(x))=\eta(x).
\end{equation}
Moreover, \eqref{eq:rec-phi-y} and \eqref{eqb3} give
\begin{equation}\label{eq:phi-bottom}
        \phi_y(x,-1)=0.
\end{equation}
Substituting \(\bar y=1\) into \eqref{eq:rec-zeta-x}, using
\eqref{eqb4} and \eqref{equitheta}, and then applying
\eqref{eq:zeta-x-minus-phi-y}--\eqref{eq:zeta-y-phi-x}, we obtain the kinematic
condition
\begin{equation}\label{eq:kinematic-reconstructed}
        \phi_y(x,\eta(x))
        =
        \bigl(\zeta_y(x,\eta(x))-1\bigr)\eta_x(x)
        =
        \phi_x(x,\eta(x))\eta_x(x)
        -\omega\eta(x)\eta_x(x)-\eta_x(x).
\end{equation}

It remains to verify the Bernoulli condition. Since \(\eta_x\) and \(\tilde z\)
have the same sign, \eqref{equieta} gives
\begin{equation}\label{eq:ztilde-eta-x}
        \tilde z
        =
        \frac{\sigma\eta_x}{\sqrt{1+\eta_x^2}}.
\end{equation}
We set
\begin{equation}\label{eq:K-def}
        K(x)
        :=
        \int_0^1
        \bar y\,\varphi_{\bar y}
        \bigl(\theta_{\bar y}+\omega(1-\bar y)-1\bigr)\,d\bar y.
\end{equation}
By the definition of \(\tilde z\) in \eqref{equizs}, differentiating in \(x\)
gives
\begin{equation}\label{eq:ztilde-differentiated}
        \left(\frac{K}{1+\eta}\right)_x
        =
        \left(\frac{\sigma\eta_x}{\sqrt{1+\eta_x^2}}\right)_x
        -z_x.
\end{equation}
On the other hand,
\begin{equation}\label{eq:K-quotient}
        \left(\frac{K}{1+\eta}\right)_x
        =
        -\frac{\eta_x}{(1+\eta)^2}K
        +\frac{1}{1+\eta}K_x.
\end{equation}
Substituting \eqref{eq:K-quotient} into
\eqref{eq:ztilde-differentiated}, and using \eqref{equiz} and
\eqref{equieta}, gives
\begin{equation}\label{eq:Bernoulli-preidentity}
\begin{aligned}
&\left(\frac{\sigma\eta_x}{\sqrt{1+\eta_x^2}}\right)_x
-\frac{1}{1+\eta}K_x \\
&\quad =
\int_0^1
\left[
\frac{(\theta_{\bar y}+\omega(1-\bar y)-1)^2-\varphi_{\bar y}^2}
{2(1+\eta)^2}
-\frac12
\right]\,d\bar y
-\int_0^1
\omega\bar y(\theta_{\bar y}+\omega(1-\bar y)-1)\,d\bar y \\
&\quad =
\frac{1}{2(1+\eta)}
\int_{-1}^{\eta}
\left((\zeta_y-1)^2-\phi_y^2\right)\,dy
-\frac12
-\frac{\omega}{1+\eta}
\int_{-1}^{\eta}(y+1)(\zeta_y-1)\,dy.
\end{aligned}
\end{equation}

We now compute \(K_x\). Since
\(\gamma\in C^1((a,b),N^{\omega,\sigma})\), differentiating
\eqref{eq:K-def} gives
\begin{equation}\label{eq:K-x-flat}
\begin{aligned}
K_x
&=
\int_0^1
\bar y\,\varphi_{x\bar y}
\bigl(\theta_{\bar y}+\omega(1-\bar y)-1\bigr)\,d\bar y
+
\int_0^1
\bar y\,\varphi_{\bar y}\theta_{x\bar y}\,d\bar y.
\end{aligned}
\end{equation}
Differentiating \eqref{eq:rec-phi-y} and \eqref{eq:rec-zeta-y} in \(x\), with
\(\bar y\) fixed, yields
\begin{align}
(1+\eta)\phi_{xy}
+\eta_x\phi_y
+\bar y\eta_x(1+\eta)\phi_{yy}
&=
\varphi_{x\bar y}, \label{eq:phi-x-ybar}\\
(1+\eta)\zeta_{xy}
+\eta_x(\zeta_y-1)
+\bar y\eta_x(1+\eta)\zeta_{yy}
&=
\theta_{x\bar y}. \label{eq:zeta-x-ybar}
\end{align}
Inserting \eqref{eq:phi-x-ybar}, \eqref{eq:zeta-x-ybar},
\eqref{eq:rec-phi-y}, and \eqref{eq:rec-zeta-y} into
\eqref{eq:K-x-flat}, and changing variables
\(\bar y=(y+1)/(1+\eta)\), gives
\begin{equation}\label{eq:K-x-physical}
\begin{aligned}
K_x
&=
\eta_x(1+\eta)\phi_y(\zeta_y-1)\big|_{y=\eta}+
\int_{-1}^{\eta}
(y+1)\phi_{xy}(\zeta_y-1)
+(y+1)\phi_y\zeta_{xy}\,dy.
\end{aligned}
\end{equation}
Using \eqref{eq:zeta-x-minus-phi-y} and \eqref{eq:zeta-y-phi-x}, we obtain
\begin{equation}\label{eq:K-x-final}
\begin{aligned}
-\frac{1}{1+\eta}K_x
&=
-\eta_x\phi_y(\zeta_y-1)\big|_{y=\eta}
-\frac{1}{1+\eta}
\int_{-1}^{\eta}
(y+1)(\zeta_{yy}+\omega)(\zeta_y-1)
-(y+1)\phi_y\phi_{yy}\,dy \\
&=
-\eta_x\phi_y(\zeta_y-1)\big|_{y=\eta}
-\frac{\omega}{1+\eta}
\int_{-1}^{\eta}(y+1)(\zeta_y-1)\,dy \\
&\quad
-\frac12
\left((\zeta_y-1)^2-\phi_y^2\right)\Big|_{y=\eta}
+\frac{1}{2(1+\eta)}
\int_{-1}^{\eta}
\left((\zeta_y-1)^2-\phi_y^2\right)\,dy .
\end{aligned}
\end{equation}
By \eqref{eq:kinematic-reconstructed},
\begin{equation}\label{eq:boundary-square}
        -\eta_x\phi_y(\zeta_y-1)\big|_{y=\eta}
        =
        -\phi_y^2(x,\eta(x)).
\end{equation}
Combining \eqref{eq:Bernoulli-preidentity}, \eqref{eq:K-x-final}, and
\eqref{eq:boundary-square}, we obtain the Bernoulli condition on \(y=\eta(x)\):
\begin{equation}\label{eq:Bernoulli-reconstructed}
        -\phi_x+\frac12\phi_x^2+\frac12\phi_y^2
        -\omega\eta\phi_x+\omega\eta+\frac{\omega^2}{2}\eta^2
        -\frac{\sigma\eta_{xx}}{(1+\eta_x^2)^{3/2}}
        =0.
\end{equation}
Together with \eqref{eq:phi-harmonic}, \eqref{eq:phi-bottom}, and
\eqref{eq:kinematic-reconstructed}, this proves that \((\phi,\eta)\) satisfies
\eqref{main fourth}. The proof is complete.
\end{proof}

\section{Fixed-domain Hamiltonian system and center-manifold reduction}\label{sec:center}
We next prepare the infinite-dimensional Hamiltonian system for the
center-manifold reduction. The point is that the domain of the Hamiltonian
vector field contains the nonlinear boundary condition \eqref{BC first}. We
remove this nonlinearity by a local change of variables.

\subsection{Removing the nonlinear boundary condition}

Fix a parameter value \((\omega_0,\sigma_0)\), and set
\begin{equation}\label{eq:center-epsilon-def}
        \varepsilon_1:=\omega-\omega_0,\qquad
        \varepsilon_2:=\sigma-\sigma_0,\qquad
        \varepsilon:=(\varepsilon_1,\varepsilon_2).
\end{equation}
We work near \(0\) and for \(|\varepsilon|\ll1\). In these parameters,
\eqref{old Hamiltons eq} becomes
\begin{equation}\label{eq:center-Hamiltonian-eps}
\begin{aligned}
\dot\varphi
&=
\frac{1}{1+\eta}
\left(
\theta_y+
\frac{\tilde z(y\varphi_y-\varphi(1))}
{\sqrt{(\sigma_0+\varepsilon_2)^2-\tilde z^{\,2}}}
\right)
+(\omega_0+\varepsilon_1)(1+\eta)\left(y-\frac12\right)
-\frac{\omega_0+\varepsilon_1}{1+\eta}\left(y-\frac12\right),
\\
\dot\theta
&=
\frac{1}{1+\eta}
\left(
\frac{\tilde z\,y(\theta_y+(\omega_0+\varepsilon_1)(1-y)-1)}
{\sqrt{(\sigma_0+\varepsilon_2)^2-\tilde z^{\,2}}}
-\varphi_y
\right),
\\
\dot z
&=
\int_0^1
\left[
\frac{(\theta_y+(\omega_0+\varepsilon_1)(1-y)-1)^2-\varphi_y^2}
{2(1+\eta)^2}
-\frac12
\right]\,dy
\\
&\quad
+
\frac{\tilde z}{\sqrt{(\sigma_0+\varepsilon_2)^2-\tilde z^{\,2}}}
\int_0^1
\frac{y\varphi_y(\theta_y+(\omega_0+\varepsilon_1)(1-y)-1)}
{(1+\eta)^2}\,dy
\\
&\quad
-\int_0^1
(\omega_0+\varepsilon_1)y
(\theta_y+(\omega_0+\varepsilon_1)(1-y)-1)\,dy,
\\
\dot\eta
&=
\frac{\tilde z}{\sqrt{(\sigma_0+\varepsilon_2)^2-\tilde z^{\,2}}},
\end{aligned}
\end{equation}
where
\begin{equation}\label{eq:center-ztilde-eps}
        \tilde z
        =
        z+
        \int_0^1
        \frac{
        y\varphi_y(\theta_y+(\omega_0+\varepsilon_1)(1-y)-1)}
        {1+\eta}\,dy .
\end{equation}
The boundary conditions are
\begin{equation}\label{eq:center-new-BC}
\begin{aligned}
        \varphi_y(0)&=0,\\
        \varphi_y(1)
        &=
        \frac{\tilde z(\theta_y(1)-1)}
        {\sqrt{(\sigma_0+\varepsilon_2)^2-\tilde z^{\,2}}}.
\end{aligned}
\end{equation}

We introduce the extended phase space
\begin{equation}\label{eq:center-extended-M}
        \widetilde M
        :=
        \left\{
        (m,\varepsilon)\in M_0\times\mathbb R^2:
        |\tilde z|<\sigma_0+\varepsilon_2,\ \eta>-1
        \right\}.
\end{equation}
Define \(f:\widetilde M\to M_0\times\mathbb R^2\) by
\begin{equation}\label{eq:center-map-f}
        f(\varphi,\theta,z,\eta,\varepsilon_1,\varepsilon_2)
        =
        (\Phi,\theta,Z,\eta,\varepsilon_1,\varepsilon_2),
\end{equation}
Here we adopt the notations
\begin{equation}\label{eq:center-Phi-Z-def}
\begin{aligned}
        \Phi
        &:=
        \varphi
        -
        W\left(
        \theta_{\mathrm{corr}}-\frac12\left(y^2-\frac13\right)
        \right),
        \\
        Z
        &:=
        -\varphi(1),
\end{aligned}
\end{equation}
where we define
\begin{equation}\label{eq:center-W-theta-p-def}
\begin{aligned}
        W
        &:=
        \frac{\tilde z}{\sqrt{(\sigma_0+\varepsilon_2)^2-\tilde z^{\,2}}},
        \\
        \theta_{\mathrm{corr}}(y)
        &:=
        \int_0^y s\,\theta_y(s)\,ds
        -
        \int_0^1\int_0^r s\,\theta_y(s)\,ds\,dr .
\end{aligned}
\end{equation}
On the target side, set
\begin{equation}\label{eq:center-R-Zp-def}
\begin{aligned}
        R
        &:=
        -\frac{\Phi(1)+Z}{\theta_{\mathrm{corr}}(1)-\frac13},
        \\
        Z_{\mathrm{corr}}
        &:=
        \int_0^1
        \frac{y}{1+\eta}
        \bigl(\Phi_y+Ry(\theta_y-1)\bigr)
        \bigl(\theta_y-1+(\omega_0+\varepsilon_1)(1-y)\bigr)\,dy .
\end{aligned}
\end{equation}

\begin{lemma}\label{lem:center-RW}
With the above definitions,
\begin{equation}\label{eq:center-RcircF}
        R\circ f=W .
\end{equation}
\end{lemma}

\begin{proof}
If \(f(\varphi,\theta,z,\eta,\varepsilon)=(\Phi,\theta,Z,\eta,\varepsilon)\),
then
\begin{equation}\label{eq:center-RW-proof}
        \Phi(1)+Z
        =
        \varphi(1)
        -
        W\left(\theta_{\mathrm{corr}}(1)-\frac13\right)
        -
        \varphi(1)
        =
        -W\left(\theta_{\mathrm{corr}}(1)-\frac13\right).
\end{equation}
Substituting this identity into the definition of \(R\) gives
\eqref{eq:center-RcircF}.
\end{proof}

\begin{proposition}[Linearization of the boundary condition]
\label{lem:center-boundary-linearization}
The change of variables \eqref{eq:center-map-f} sends the nonlinear boundary
condition \eqref{eq:center-new-BC} to
\begin{equation}\label{eq:center-linear-BC}
        \Phi_y(0)=\Phi_y(1)=0 .
\end{equation}
\end{proposition}

\begin{proof}
From \eqref{eq:center-Phi-Z-def},
\begin{equation}\label{eq:center-Phi-y}
        \Phi_y
        =
        \varphi_y
        -
        W\,y(\theta_y-1).
\end{equation}
At \(y=0\), this gives \(\Phi_y(0)=\varphi_y(0)=0\). At \(y=1\), using
\eqref{eq:center-new-BC} and the definition of \(W\), we obtain
\begin{equation}\label{eq:center-Phi-y-one}
        \Phi_y(1)
        =
        \varphi_y(1)
        -
        \frac{\tilde z(\theta_y(1)-1)}
        {\sqrt{(\sigma_0+\varepsilon_2)^2-\tilde z^{\,2}}}
        =
        0.
\end{equation}
\end{proof}

\begin{proposition}[Boundary-linearizing diffeomorphism]
\label{prop:center-boundary-diffeomorphism}
The map \(f\) is smooth, satisfies \(f(0)=0\), and is a local diffeomorphism near
the origin. More precisely, there are neighborhoods \(U_1,U_2\subset M_0\) and
\(\mathcal P\subset\mathbb R^2\) of \(0\) such that
\begin{equation}\label{eq:center-f-diffeo}
        f:U_1\times\mathcal P\longrightarrow U_2\times\mathcal P
\end{equation}
is a smooth diffeomorphism. Its inverse is
\begin{equation}\label{eq:center-f-inverse}
\begin{aligned}
f^{-1}(\Phi,\theta,Z,\eta,\varepsilon_1,\varepsilon_2)
=
\Bigg(
&\Phi
+
R\left(
\theta_{\mathrm{corr}}-\frac12\left(y^2-\frac13\right)
\right),
\theta, \frac{(\sigma_0+\varepsilon_2)R}{\sqrt{1+R^2}}-Z_{\mathrm{corr}},
\eta,\varepsilon_1,\varepsilon_2
\Bigg).
\end{aligned}
\end{equation}
\end{proposition}

\begin{proof}
The smoothness of \(f\) follows from the definitions and the inequality
\(|\tilde z|<\sigma_0+\varepsilon_2\). A direct computation gives
\begin{equation}\label{eq:center-df0}
\begin{aligned}
df(0)(\varphi,\theta,z,\eta,\varepsilon_1,\varepsilon_2)
=
\Bigg(
&\varphi
+
\frac{1}{2\sigma_0}
\left(y^2-\frac13\right)
\left(
z-\varphi(1)+2\omega_0\int_0^1 y\varphi\,dy
\right),
\\
&\theta,\ -\varphi(1),\ \eta,\ \varepsilon_1,\ \varepsilon_2
\Bigg),
\end{aligned}
\end{equation}
with inverse
\begin{equation}\label{eq:center-df0-inverse}
\begin{aligned}
(df(0))^{-1}(\Phi,\theta,Z,\eta,\varepsilon_1,\varepsilon_2)
=
\Bigg(
&\Phi
-\frac32\left(y^2-\frac13\right)(\Phi(1)+Z),
\theta,
\\
&\left(3\sigma_0-1+\frac{\omega_0}{4}\right)(\Phi(1)+Z)
+\Phi(1)-2\omega_0\int_0^1 y\Phi\,dy,
\eta,\varepsilon_1,\varepsilon_2
\Bigg).
\end{aligned}
\end{equation}
Thus \(df(0)\) is a bounded linear isomorphism. The inverse function theorem
\cite[\S~2.5]{Abraham1988} gives the local diffeomorphism
\eqref{eq:center-f-diffeo}. Formula \eqref{eq:center-f-inverse} is verified by
direct substitution, using Lemma~\ref{lem:center-RW}.
\end{proof}

Let \(\imath:N_0\hookrightarrow M_0\) denote the embedding. Define
\begin{equation}\label{eq:center-tilde-f-def}
        \tilde f_1^\varepsilon:=f_1^\varepsilon\circ\imath,
        \qquad
        \tilde f(m,\varepsilon):=(\tilde f_1^\varepsilon(m),\varepsilon),
\end{equation}
and set
\begin{equation}\label{eq:center-tilde-N-def}
        \widetilde N
        :=
        \left\{
        (m,\varepsilon)\in N_0\times\mathbb R^2:
        |\tilde z|<\sigma_0+\varepsilon_2,\ \eta>-1
        \right\}.
\end{equation}

\begin{proposition}[Restriction to the \(X^2\)-domain]
\label{prop:center-restricted-diffeomorphism}
The map \(\tilde f\) is a smooth local diffeomorphism from a neighborhood
\(U_3\times\mathcal P\) of the origin in \(\widetilde N\) onto a neighborhood
\(U_4\times\mathcal P\) of the origin in \(N_0\times\mathbb R^2\). Its inverse
is the restriction of \eqref{eq:center-f-inverse}.
\end{proposition}

\begin{proof}
The formula \eqref{eq:center-f-inverse} preserves the \(N_0\)-regularity near
the origin. Hence
\begin{equation}\label{eq:center-restricted-df}
        d\tilde f_1^\varepsilon(0)=df_1^\varepsilon(0)\circ\imath,
        \qquad
        (d\tilde f_1^\varepsilon(0))^{-1}
        =(df_1^\varepsilon(0))^{-1}\circ\imath
\end{equation}
are bounded linear operators in the \(X^2\)-topology. The inverse function
theorem gives the claim.
\end{proof}

\begin{remark}\label{rem:center-two-topologies}
The local diffeomorphisms \(f_1^\varepsilon\) and \(\tilde f_1^\varepsilon\)
are used with different topologies. For fixed \(\varepsilon\),
\(f_1^\varepsilon\) is controlled in the \(X^1\)-topology, while
\(\tilde f_1^\varepsilon\) is controlled in the \(X^2\)-topology. This reflects
the mapping property
\begin{equation}\label{eq:center-vector-field-mapping}
        v_H^{\omega,\sigma}:N^{\omega,\sigma}\subset X^2\longrightarrow X^1.
\end{equation}
Thus the change of variables must be smooth on the \(X^2\)-domain, while the
transformed vector field is estimated in the \(X^1\)-target space.
\end{remark}

\subsection{The transformed Hamiltonian structure}

We now conjugate the Hamiltonian system by \(f_1^\varepsilon\). Define
\begin{equation}\label{eq:center-H1-def}
        H_1^\varepsilon(\varphi,\theta,z,\eta)
        :=
        H^{\omega_0+\varepsilon_1,\sigma_0+\varepsilon_2}
        (\varphi,\theta,z,\eta).
\end{equation}
By \eqref{second Hamiltonian},
\begin{equation}\label{eq:center-H1-formula}
\begin{aligned}
H_1^\varepsilon
&=
\int_0^1
\frac{
(\theta_y+(\omega_0+\varepsilon_1)(1-y)+\eta)^2-\varphi_y^2}
{2(1+\eta)}\,dy
\\
&\quad
+\int_0^1
(\omega_0+\varepsilon_1)\bigl(y(1+\eta)-1\bigr)
\bigl(\theta_y-1+(\omega_0+\varepsilon_1)(1-y)\bigr)\,dy
\\
&\quad
-\frac12(\omega_0+\varepsilon_1)
+\frac16(\omega_0+\varepsilon_1)^2
+\sigma_0+\varepsilon_2
-\sqrt{(\sigma_0+\varepsilon_2)^2-\tilde z^{\,2}} .
\end{aligned}
\end{equation}
The Hamiltonian in the new variables is
\begin{equation}\label{eq:center-Htilde-def}
        \tilde H^\varepsilon(\Phi,\theta,Z,\eta)
        :=
        H_1^\varepsilon
        \bigl((f_1^\varepsilon)^{-1}(\Phi,\theta,Z,\eta)\bigr).
\end{equation}
Using \eqref{eq:center-f-inverse}, one obtains
\begin{equation}\label{eq:center-Htilde-formula}
\begin{aligned}
\tilde H^\varepsilon
&=
\frac{1}{2(1+\eta)}
\int_0^1
\theta_y^2-\bigl(\Phi_y+Ry(\theta_y-1)\bigr)^2\,dy
+\frac{\eta^2}{2(1+\eta)}
\\
&\quad
+\frac{\omega_0+\varepsilon_1}{1+\eta}\int_0^1\theta\,dy
-(\omega_0+\varepsilon_1)(1+\eta)\int_0^1\theta\,dy
\\
&\quad
+\frac{(\omega_0+\varepsilon_1)^2}{6(1+\eta)}
+\frac{(\omega_0+\varepsilon_1)\eta}{2(1+\eta)}
+\frac16(\omega_0+\varepsilon_1)^2\eta
-\frac16(\omega_0+\varepsilon_1)^2
\\
&\quad
-\frac12(\omega_0+\varepsilon_1)\eta
+\sigma_0+\varepsilon_2
-\frac{\sigma_0+\varepsilon_2}{\sqrt{1+R^2}} .
\end{aligned}
\end{equation}
The symplectic form transforms by pullback:
\begin{equation}\label{eq:center-Omega-tilde}
\begin{aligned}
\tilde\Omega_m^\varepsilon[v^1,v^2]
:=
\Omega^{\omega_0+\varepsilon_1,\sigma_0+\varepsilon_2}_{(f_1^\varepsilon)^{-1}(m)}
\left[
(d f_1^\varepsilon)^{-1}_m[v^1],
(d f_1^\varepsilon)^{-1}_m[v^2]
\right].
\end{aligned}
\end{equation}

By Lemma~\ref{lem:center-boundary-linearization}, we introduce the fixed domain
\begin{equation}\label{eq:center-domain-L}
        \operatorname{dom}(\mathbf L)
        :=
        \left\{
        (\Phi,\theta,Z,\eta)\in N_0:
        \Phi_y(0)=\Phi_y(1)=0
        \right\},
\end{equation}
equipped with the \(X^2\)-norm.

\begin{proposition}[Transformed Hamiltonian vector field]
\label{prop:center-transformed-vector-field}
For \(m\in\operatorname{dom}(\mathbf L)\) sufficiently small in \(X^2\), the
Hamiltonian vector field associated with
\((\tilde H^\varepsilon,\tilde\Omega^\varepsilon)\) is
\begin{equation}\label{eq:center-new-vH}
        \widetilde{v_H^\varepsilon}(m)
        =
        d f_1^\varepsilon((f_1^\varepsilon)^{-1}(m))
        \left[
        v_H^{\omega_0+\varepsilon_1,\sigma_0+\varepsilon_2}
        ((f_1^\varepsilon)^{-1}(m))
        \right].
\end{equation}
Writing
\(\widetilde{v_H^\varepsilon}
=(\widetilde{\Phi_H^\varepsilon},
\widetilde{\theta_H^\varepsilon},
\widetilde{Z_H^\varepsilon},
\widetilde{\eta_H^\varepsilon})\), one has
\begin{equation}\label{eq:center-explicit-vH}
\begin{aligned}
\widetilde{\Phi_H^\varepsilon}(m)
&=
\frac{1}{1+\eta}
\Bigl[
\theta_y
+R\bigl(2y\Phi_y+Ry^2(\theta_y-1)-\Phi-\Phi(1)+Z\bigr)
\Bigr]
\\
&\quad
-\frac{\bar Z(1+R^2)^{3/2}}{\sigma_0+\varepsilon_2}
\left(
\theta_{\mathrm{corr}}-\frac12\left(y^2-\frac13\right)
\right)
+\frac{(\omega_0+\varepsilon_1)R^2y^2(4y-3)}
{6(1+\eta)}
\\
&\quad
+(\omega_0+\varepsilon_1)(1+\eta)\left(y-\frac12\right)
-\frac{\omega_0+\varepsilon_1}{1+\eta}\left(y-\frac12\right),
\\
\widetilde{\theta_H^\varepsilon}(m)
&=
-\frac{1}{1+\eta}
\left(\Phi_y-(\omega_0+\varepsilon_1)R(1-y)y\right),
\\
\widetilde{Z_H^\varepsilon}(m)
&=
-\frac{1}{1+\eta}
\Bigl(\theta_y(1)+R\bigl(R(\theta_y(1)-1)+Z\bigr)\Bigr)
-\frac{(\omega_0+\varepsilon_1)(1+\eta)}{2}
+\frac{\omega_0+\varepsilon_1}{2(1+\eta)},
\\
\widetilde{\eta_H^\varepsilon}(m)
&=
R,
\end{aligned}
\end{equation}
where
\begin{equation}\label{eq:center-Zbar-def}
        \bar Z
        =
        \frac{(1+R^2)(\theta_y(1)-1)^2}{2(1+\eta)^2}
        -\frac12 .
\end{equation}
Consequently, near the origin, a curve
\(\gamma\in C^1((a,b),N^{\omega,\sigma})\) solves \eqref{old Hamiltons eq} with
\(\gamma((a,b))\subset\operatorname{dom}(v_H^{\omega,\sigma})\) if and only if
\begin{equation}\label{eq:center-v-def}
        v:=\tilde f_1^\varepsilon\circ\gamma
        \in C^1((a,b),\operatorname{dom}(\mathbf L))
\end{equation}
solves
\begin{equation}\label{eq:center-transformed-system}
        \dot v(x)=\widetilde{v_H^\varepsilon}(v(x)).
\end{equation}
\end{proposition}

\begin{proof}
For \(v_m'\in M_0\), the chain rule and the definition of
\(\tilde\Omega^\varepsilon\) give
\begin{equation}\label{eq:center-chain-rule-vH}
\begin{aligned}
d\tilde H_m^\varepsilon[v_m']
&=
dH^{\omega_0+\varepsilon_1,\sigma_0+\varepsilon_2}_{(f_1^\varepsilon)^{-1}(m)}
\bigl[(df_1^\varepsilon)^{-1}_m[v_m']\bigr]
\\
&=
\tilde\Omega_m^\varepsilon
\left[
d f_1^\varepsilon((f_1^\varepsilon)^{-1}(m))
v_H^{\omega_0+\varepsilon_1,\sigma_0+\varepsilon_2}
((f_1^\varepsilon)^{-1}(m)),
v_m'
\right].
\end{aligned}
\end{equation}
This proves \eqref{eq:center-new-vH}. Formula
\eqref{eq:center-explicit-vH} follows by direct substitution using
\eqref{eq:center-f-inverse}; the computation is analogous to
\cite[Appendix~A]{buffoni1996plethora}. The equivalence of
\eqref{old Hamiltons eq} and \eqref{eq:center-transformed-system} follows from
\eqref{eq:center-new-vH} and the diffeomorphism property of
\(\tilde f_1^\varepsilon\).
\end{proof}

In particular, it remains to analyze the transformed system
\eqref{eq:center-transformed-system} on the fixed domain
\(\operatorname{dom}(\mathbf L)\).

\subsection{The linearized operator \(\mathbf L\)}

The symplectic form \(\tilde\Omega^\varepsilon\) and the vector field
\(\widetilde{v_H^\varepsilon}\) constructed above depend smoothly on
\(\varepsilon=(\varepsilon_1,\varepsilon_2)\). Thus, near \((0,0)\), we regard
\begin{equation}\label{eq:center-vH-smooth-map}
        \widetilde{v_H}:
        \operatorname{dom}(\mathbf L)\times\mathbb R^2\longrightarrow M_0,
        \qquad
        (v,\varepsilon)\longmapsto \widetilde{v_H^\varepsilon}(v),
\end{equation}
as a smooth map. Its first-order expansion in the \(v\)-variable gives
\begin{equation}\label{eq:center-v-L-plus-B}
        \dot v=\mathbf L v+\mathbf B(v,\varepsilon),
        \qquad
        \mathbf L:=d_1\widetilde{v_H}(0,0),
        \qquad
        \mathbf B(v,\varepsilon):=\widetilde{v_H}(v,\varepsilon)-\mathbf L v .
\end{equation}
The remainder \(\mathbf B\) is smooth near \((0,0)\), and
\begin{equation}\label{eq:center-B-remainder}
        \mathbf B(v,0)=O(\|v\|_{X^2}^2)
        \quad\text{as }v\to0,
        \qquad
        d_1\mathbf B(0,0)=0 .
\end{equation}

A direct calculation from \eqref{eq:center-explicit-vH} gives, for
\(v=(\Phi,\theta,Z,\eta)\),
\begin{equation}\label{eq:center-L-formula}
\mathbf Lv
=
\left(
\begin{aligned}
&
\theta_y
+\frac{1}{2\sigma_0}\bigl(\theta_y(1)+\eta\bigr)
\left(\frac13-y^2\right)
+\omega_0\eta(2y-1),
\\
&
-\Phi_y+3\omega_0(\Phi(1)+Z)(1-y)y,
\\
&
-\theta_y(1)-\omega_0\eta,
\\
&
3(\Phi(1)+Z)
\end{aligned}
\right).
\end{equation}
We view \(\mathbf L\) as an unbounded operator
\begin{equation}\label{eq:center-L-unbounded}
        \mathbf L:\operatorname{dom}(\mathbf L)\subset M_0\longrightarrow M_0,
\end{equation}
where \(\operatorname{dom}(\mathbf L)\) is defined in \eqref{eq:center-domain-L}.

\begin{lemma}\label{lembfL}
The operator \(\mathbf L:\operatorname{dom}(\mathbf L)\subset M_0\to M_0\) is
densely defined and closed. Moreover, its graph norm is equivalent to the
\(X^2\)-norm on \(\operatorname{dom}(\mathbf L)\): there exists \(C\ge1\) such
that
\begin{equation}\label{eq:center-graph-norm-equivalence}
        C^{-1}\|v\|_{X^2}
        \le
        \|v\|_{G(\mathbf L)}
        :=
        \|v\|_{X^1}+\|\mathbf Lv\|_{X^1}
        \le
        C\|v\|_{X^2},
        \qquad
        v\in\operatorname{dom}(\mathbf L).
\end{equation}
\end{lemma}

\begin{proof}
The density of \(\operatorname{dom}(\mathbf L)\) in \(M_0\) follows by the
standard density of smooth functions satisfying the defining linear constraints.

We prove closedness. Let
\(v_n=(\Phi_n,\theta_n,Z_n,\eta_n)\in\operatorname{dom}(\mathbf L)\) satisfy
\begin{equation}\label{eq:center-closedness-assumption}
        v_n\to v=(\Phi,\theta,Z,\eta),
        \qquad
        \mathbf Lv_n\to w
        \qquad
        \text{in }X^1 .
\end{equation}
From \eqref{eq:center-L-formula}, the convergence of \(\mathbf Lv_n\) in
\(X^1\) implies that \((\Phi_n)_y\) and \((\theta_n)_y\) converge in \(H^1(0,1)\)
to some \(\Phi_*\) and \(\theta_*\), respectively. For
\(\psi\in C_c^\infty(0,1)\),
\begin{equation}\label{eq:center-closedness-distribution}
\begin{aligned}
\int_0^1\Phi\,\psi_{yy}\,dy
&=
\lim_{n\to\infty}\int_0^1\Phi_n\,\psi_{yy}\,dy
=
\lim_{n\to\infty}\int_0^1(\Phi_n)_{yy}\,\psi\,dy =
\int_0^1(\Phi_*)_y\,\psi\,dy .
\end{aligned}
\end{equation}
Hence \(\Phi_{yy}=(\Phi_*)_y\in L^2(0,1)\), so \(\Phi\in H^2(0,1)\). Since
\(\Phi_n\to\Phi\) in \(H^1\) and \((\Phi_n)_y\to\Phi_*\) in \(H^1\), we have $\Phi_n\to\Phi$ in $H^2(0,1)$. The trace theorem and the conditions \((\Phi_n)_y(0)=(\Phi_n)_y(1)=0\) imply
\begin{equation}\label{eq:center-Phi-boundary-limit}
        \Phi_y(0)=\Phi_y(1)=0 .
\end{equation}
The same argument gives
\begin{equation}\label{eq:center-theta-H2-convergence}
        \theta\in H^2(0,1),
        \qquad
        \theta_n\to\theta
        \quad\text{in }H^2(0,1).
\end{equation}
Thus \(v\in\operatorname{dom}(\mathbf L)\). Passing to the limit in
\eqref{eq:center-L-formula} gives \(w=\mathbf Lv\), so \(\mathbf L\) is closed.

We now prove \eqref{eq:center-graph-norm-equivalence}. The upper bound follows
directly from \eqref{eq:center-L-formula} and the trace estimate
\eqref{traceb}: there is \(C_*\ge1\) such that
\begin{equation}\label{eq:center-graph-upper-bound}
        \|v\|_{G(\mathbf L)}\le C_*\|v\|_{X^2},
        \qquad
        v\in\operatorname{dom}(\mathbf L).
\end{equation}
For the reverse bound, consider the identity map
\begin{equation}\label{eq:center-identity-map-graph}
        I:
        \bigl(\operatorname{dom}(\mathbf L),\|\cdot\|_{X^2}\bigr)
        \longrightarrow
        \bigl(\operatorname{dom}(\mathbf L),\|\cdot\|_{G(\mathbf L)}\bigr).
\end{equation}
By \eqref{eq:center-graph-upper-bound}, \(I\) is continuous. Since \(\mathbf L\)
is closed, \(\operatorname{dom}(\mathbf L)\) equipped with the graph norm is a
Banach space; equipped with the \(X^2\)-norm it is also Banach. The bounded
inverse theorem \cite[\S~12.5]{Conway1996} implies that \(I^{-1}\) is
continuous. Therefore
\begin{equation}\label{eq:center-graph-lower-bound}
        \|v\|_{X^2}\le C\|v\|_{G(\mathbf L)},
        \qquad
        v\in\operatorname{dom}(\mathbf L),
\end{equation}
which completes the proof.
\end{proof}

By Lemma~\ref{lembfL}, the \(X^2\)-regularity on
\(\operatorname{dom}(\mathbf L)\) is equivalent to graph-norm regularity. Let
\begin{equation}\label{eq:center-jmath-def}
        \jmath:\operatorname{dom}(\mathbf L)\hookrightarrow M_0
\end{equation}
denote the inclusion, with \(\operatorname{dom}(\mathbf L)\) equipped with the
\(X^2\)-norm. Define
\begin{equation}\label{eq:center-Hhat-def}
        \hat H^\varepsilon:=\tilde H^\varepsilon\circ\jmath .
\end{equation}
Then \(\hat H^\varepsilon\) is smooth on \(\operatorname{dom}(\mathbf L)\).

\begin{lemma}\label{lemequ}
For every \(u^1,u^2\in\operatorname{dom}(\mathbf L)\),
\begin{equation}\label{eq:center-Hessian-symplectic-identity}
        d^2\hat H^0_0[u^1,u^2]
        =
        \tilde\Omega^0_0[\mathbf L u^1,\jmath(u^2)] .
\end{equation}
\end{lemma}

\begin{proof}
By definition of the transformed Hamiltonian vector field, for \(m\) near \(0\)
and \(v'\in M_0\),
\begin{equation}\label{eq:center-Hamiltonian-field-identity}
        d\tilde H^\varepsilon_m[v']
        =
        \tilde\Omega^\varepsilon_m
        [\widetilde{v_H^\varepsilon}(m),v'] .
\end{equation}
Set \(\varepsilon=0\) and differentiate \eqref{eq:center-Hamiltonian-field-identity}
at \(m=0\) in the direction \(u^1\in\operatorname{dom}(\mathbf L)\). Since
\(\widetilde{v_H^0}(0)=0\), the derivative of the base point in
\(\tilde\Omega^0_m\) does not contribute, and
\begin{equation}\label{eq:center-d2Htilde-L}
        d^2\tilde H^0_0[u^1,v']
        =
        \tilde\Omega^0_0[\mathbf L u^1,v'],
        \qquad
        v'\in M_0.
\end{equation}
Since \(\hat H^0=\tilde H^0\circ\jmath\), we also have
\begin{equation}\label{eq:center-Hhat-Htilde-Hessian}
        d^2\hat H^0_0[u^1,u^2]
        =
        d^2\tilde H^0_0[u^1,\jmath(u^2)] .
\end{equation}
Taking \(v'=\jmath(u^2)\) in \eqref{eq:center-d2Htilde-L} and using
\eqref{eq:center-Hhat-Htilde-Hessian} gives
\eqref{eq:center-Hessian-symplectic-identity}.
\end{proof}

We finally record the symmetry inherited by the transformed system.

\begin{proposition}[Reversibility after the change of variables]
\label{prop:center-reversibility}
The involution \(S(\Phi,\theta,Z,\eta)=(-\Phi,\theta,-Z,\eta)\) preserves
\(\operatorname{dom}(\mathbf L)\), and the transformed Hamiltonian system is
reversible with respect to \(S\). More precisely,
\begin{align}
        \tilde H^\varepsilon(Sm)
        &=
        \tilde H^\varepsilon(m), \label{eq:center-symH2}\\
        \tilde\Omega^\varepsilon_m[v^1,v^2]
        &=
        -\tilde\Omega^\varepsilon_{Sm}[Sv^1,Sv^2],
        \qquad v^1,v^2\in M_0, \label{eq:center-symOmega2}\\
        \widetilde{v_H^\varepsilon}(Sm)
        &=
        -S\widetilde{v_H^\varepsilon}(m).
        \label{eq:center-symvH2}
\end{align}
\end{proposition}

\begin{proof}
The formulas defining \(f_1^\varepsilon\) and \((f_1^\varepsilon)^{-1}\) give
\begin{equation}\label{eq:center-f-commutes-S}
        f_1^\varepsilon\circ S=S\circ f_1^\varepsilon,
        \qquad
        (f_1^\varepsilon)^{-1}\circ S
        =
        S\circ(f_1^\varepsilon)^{-1}.
\end{equation}
Differentiating \eqref{eq:center-f-commutes-S} yields
\begin{equation}\label{eq:center-df-commutes-S}
        df_1^\varepsilon(Sm)\circ S
        =
        S\circ df_1^\varepsilon(m),
        \qquad
        (df_1^\varepsilon)^{-1}(Sm)\circ S
        =
        S\circ (df_1^\varepsilon)^{-1}(m).
\end{equation}
The identities \eqref{eq:center-symH2} and \eqref{eq:center-symOmega2} follow
from \eqref{eq:center-Htilde-def}, \eqref{eq:center-Omega-tilde},
\eqref{eq:symH}, \eqref{eq:symOmega}, and
\eqref{eq:center-f-commutes-S}--\eqref{eq:center-df-commutes-S}. Finally,
\eqref{eq:center-symvH2} follows from the transformed vector-field formula
\eqref{eq:center-new-vH} and the original reversibility identity
\eqref{eq:symvH}.
\end{proof}

\subsection{Hamiltonian center-manifold theorem and Darboux coordinates}

We use Mielke's Hamiltonian center-manifold theorem
\cite{mielke1991hamiltonian}, in the form used in water-wave spatial dynamics;
see also \cite{buffoni1996plethora,buffoni1999multiplicity}. We state only the
version needed below.

\begin{theorem}[Hamiltonian center manifold and Darboux coordinates]
\label{thm:Mielke_center_Darboux}
Consider the reversible Hamiltonian equation
\begin{equation}\label{eq:Mielke-abstract-equation}
        \dot u=\mathsf L u+\mathsf B(u,\epsilon),
        \qquad
        u\in\mathsf M,\quad \epsilon\in\mathbb R^\ell,
\end{equation}
where \(\mathsf M\) is a Hilbert space,
\(\mathsf L:\operatorname{dom}(\mathsf L)\subset\mathsf M\to\mathsf M\) is a
densely defined closed operator, and \(\operatorname{dom}(\mathsf L)\) is
equipped with the graph norm. Assume that \(u=0\) is an equilibrium at
\(\epsilon=0\), and that the following hypotheses hold:
\begin{enumerate}[label=(H\arabic*),leftmargin=2.4em]
\item The part of \(\sigma(\mathsf L)\) on the imaginary axis consists of
finitely many eigenvalues of finite algebraic multiplicity and is separated
from the rest of the spectrum. If \(\mathsf P\) denotes the corresponding
spectral projection, then
\begin{equation}\label{eq:Mielke-spectral-splitting}
        \mathsf M=\mathsf M_1\oplus\mathsf M_2,
        \qquad
        \mathsf M_1:=\mathsf P\mathsf M,\qquad
        \mathsf M_2:=(I-\mathsf P)\mathsf M .
\end{equation}

\item With
\(\mathsf L_2:=\mathsf L|_{\mathsf M_2\cap\operatorname{dom}(\mathsf L)}\),
one has
\begin{equation}\label{eq:Mielke-resolvent}
        \|(\mathsf L_2-i\alpha I)^{-1}\|_{\mathsf M_2\to\mathsf M_2}
        \le \frac{C}{1+|\alpha|},
        \qquad \alpha\in\mathbb R .
\end{equation}

\item There are neighborhoods \(\mathsf U\subset\operatorname{dom}(\mathsf L)\)
and \(\mathsf E\subset\mathbb R^\ell\) of \(0\), and an integer \(k\ge2\), such
that \(\mathsf B\in C^{k+1}(\mathsf U\times\mathsf E;\mathsf M)\), with bounded
and uniformly continuous derivatives, and
\begin{equation}\label{eq:Mielke-B-vanishing}
        \mathsf B(0,0)=0,\qquad d_1\mathsf B(0,0)=0 .
\end{equation}

\item The equation is Hamiltonian with respect to
\((\mathsf M,\mathsf\Omega^\epsilon,\mathsf H^\epsilon)\), where
\(\mathsf H^\epsilon\) is \(C^k\) and \(\mathsf\Omega^\epsilon\) is \(C^{k-1}\)
in both the phase variable and the parameter.

\item There is an isometric reverser \(\mathsf S:\mathsf M\to\mathsf M\) with
\(\mathsf S^2=I\), \(\mathsf S(\operatorname{dom}(\mathsf L))\subset
\operatorname{dom}(\mathsf L)\), and
\begin{align}
        \mathsf H^\epsilon(\mathsf S u)
        &=
        \mathsf H^\epsilon(u), \label{eq:Mielke-rev-H}\\
        \mathsf\Omega^\epsilon_{\mathsf S u}[\mathsf S v,\mathsf S w]
        &=
        -\mathsf\Omega^\epsilon_u[v,w].
        \label{eq:Mielke-rev-Omega}
\end{align}
\end{enumerate}

Then, after shrinking neighborhoods, there exists a reduction function
\begin{equation}\label{eq:Mielke-reduction-map}
        \mathsf r:\widetilde{\mathsf U}_1\times\widetilde{\mathsf E}
        \longrightarrow \widetilde{\mathsf U}_2,
        \qquad
        \widetilde{\mathsf U}_i\subset\mathsf M_i,
\end{equation}
of class \(C^k\), with bounded and uniformly continuous derivatives, such that
\begin{equation}\label{eq:Mielke-r-vanishing}
        \mathsf r(0,0)=0,\qquad d_1\mathsf r(0,0)=0 .
\end{equation}
For each \(\epsilon\in\widetilde{\mathsf E}\), the graph
\begin{equation}\label{eq:Mielke-center-manifold}
        \widetilde{\mathsf M}^{\epsilon}
        :=
        \{u_1+\mathsf r(u_1,\epsilon):
        u_1\in\widetilde{\mathsf U}_1\}
\end{equation}
is a locally invariant center manifold. Every sufficiently small bounded
solution of \eqref{eq:Mielke-abstract-equation} lies on
\(\widetilde{\mathsf M}^{\epsilon}\), and the reduced equation is
\begin{equation}\label{eq:Mielke-reduced-equation}
        \dot u_1
        =
        \mathsf L u_1
        +
        \mathsf P\mathsf B(u_1+\mathsf r(u_1,\epsilon),\epsilon).
\end{equation}
Conversely, a solution \(u_1\) of \eqref{eq:Mielke-reduced-equation} gives a
solution of the full equation by
\begin{equation}\label{eq:Mielke-lift}
        u(x)=u_1(x)+\mathsf r(u_1(x),\epsilon).
\end{equation}

Moreover, \(\widetilde{\mathsf M}^{\epsilon}\) is a symplectic submanifold, and
the reduced flow is Hamiltonian with Hamiltonian and symplectic form obtained
by restriction. The center space and its complement are invariant under
\(\mathsf S\), and the reduction function is equivariant:
\begin{equation}\label{eq:Mielke-r-equivariant}
        \mathsf r(\mathsf S u_1,\epsilon)
        =
        \mathsf S\mathsf r(u_1,\epsilon).
\end{equation}

Finally, there is a \(C^{k-1}\) parameter-dependent Darboux change of variables
on the center manifold,
\begin{equation}\label{eq:Mielke-Darboux-map}
        \tilde u_1=u_1+\Theta(u_1,\epsilon),
        \qquad
        \Theta(0,0)=0,\qquad d_1\Theta(0,0)=0,
\end{equation}
which transforms the reduced symplectic form into the constant form
\begin{equation}\label{eq:Mielke-constant-symplectic}
        \Psi[v,w]
        :=
        \widetilde{\mathsf\Omega}^{0}_{0}[v,w].
\end{equation}
The Darboux map may be chosen \(\mathsf S\)-equivariant:
\begin{equation}\label{eq:Mielke-Theta-equivariant}
        \Theta(\mathsf S u_1,\epsilon)=\mathsf S\Theta(u_1,\epsilon).
\end{equation}
In the corresponding coordinates $\tilde\chi^{-1}(\tilde u_1) = \tilde u_1+\tilde{\mathsf r}(\tilde u_1,\epsilon)$, one has
\begin{equation}\label{eq:Mielke-rtilde-equivariant}
        \tilde{\mathsf r}(\mathsf S\tilde u_1,\epsilon)
        =
        \mathsf S\tilde{\mathsf r}(\tilde u_1,\epsilon).
\end{equation}
After choosing a symplectic basis of \(\mathsf M_1\), \(\Psi\) becomes the
canonical symplectic form, and the reduced dynamics are a finite-dimensional
canonical reversible Hamiltonian system.
\end{theorem}

\begin{proof}
The center-manifold and Hamiltonian-reduction statements are the standard
theorems of Mielke~\cite{mielke1991hamiltonian}; the parameter-dependent
Darboux statement is the theorem of Buffoni--Groves
\cite[Theorem~4]{buffoni1999multiplicity}. We only record why the Darboux map
may be chosen to preserve the reverser.

By \eqref{eq:Mielke-r-equivariant}, the pulled-back reduced symplectic form
satisfies
\begin{equation}\label{eq:Mielke-reduced-Omega-anti}
        \mathsf S^*\widetilde{\mathsf\Omega}^{\epsilon}
        =
        -\widetilde{\mathsf\Omega}^{\epsilon}.
\end{equation}
Indeed, for \(v,w\in\mathsf M_1\),
\begin{equation}\label{eq:Mielke-reduced-Omega-anti-expanded}
\begin{aligned}
\widetilde{\mathsf\Omega}^{\epsilon}_{\mathsf S u_1}
[\mathsf S v,\mathsf S w]
&=
\mathsf\Omega^\epsilon_{\mathsf S(u_1+\mathsf r(u_1,\epsilon))}
\left[
\mathsf S(v+d_1\mathsf r(u_1,\epsilon)[v]),
\mathsf S(w+d_1\mathsf r(u_1,\epsilon)[w])
\right]  \\
&=
-\mathsf\Omega^\epsilon_{u_1+\mathsf r(u_1,\epsilon)}
\left[
v+d_1\mathsf r(u_1,\epsilon)[v],
w+d_1\mathsf r(u_1,\epsilon)[w]
\right] \\
&=
-\widetilde{\mathsf\Omega}^{\epsilon}_{u_1}[v,w].
\end{aligned}
\end{equation}

In the proof of the Darboux theorem, one chooses a one-form \(\alpha^0\) with $d\alpha^0 = \widetilde{\mathsf\Omega}^{\epsilon} - \widetilde{\mathsf\Omega}^{0}_{0}$. We note that replacing it by $\widetilde\alpha^0:=(\alpha^0-\mathsf S^*\alpha^0)/2$ does not change its exterior derivative, since both \(\widetilde{\mathsf\Omega}^{\epsilon}\) and
\(\widetilde{\mathsf\Omega}^{0}_{0}\) are anti-invariant under \(\mathsf S\). Moreover,
\begin{equation}\label{eq:Mielke-alpha-anti}
        \mathsf S^*\widetilde\alpha^0=-\widetilde\alpha^0.
\end{equation}
The Moser vector field used in the Darboux construction is defined by
\begin{equation}\label{eq:Mielke-Moser-field}
        \iota_{X_t}
        \bigl(
        \widetilde{\mathsf\Omega}^{0}_{0}
        +t(\widetilde{\mathsf\Omega}^{\epsilon}
        -\widetilde{\mathsf\Omega}^{0}_{0})
        \bigr)
        =
        -\widetilde\alpha^0 .
\end{equation}
Equations \eqref{eq:Mielke-reduced-Omega-anti} and
\eqref{eq:Mielke-alpha-anti} imply
\begin{equation}\label{eq:Mielke-Moser-equivariant}
        X_t(\mathsf S u_1)=\mathsf S X_t(u_1).
\end{equation}
Hence the Moser flow is \(\mathsf S\)-equivariant, and so the resulting Darboux
correction \(\Theta\) satisfies \eqref{eq:Mielke-Theta-equivariant}. This gives
\eqref{eq:Mielke-rtilde-equivariant}. 
\end{proof}

\section{Spectral reduction and the KdV normal form}\label{sec:normalform}

We now fix \(\omega_0=1\) and \(\sigma_0>1/3\). In the two-parameter
formulation, \(\varepsilon_1\) measures the departure of the vorticity from the
critical value \(1\), while \(\varepsilon_2\) measures the departure of the
surface-tension parameter from the fixed value \(\sigma_0\). For the solitary
waves constructed in Theorem~\ref{thm:main}, we set \(\varepsilon_2=0\) and
write \(\varepsilon=(\varepsilon_1,0)\), where \(|\varepsilon_1|\ll1\).

Our goal is to construct a homoclinic orbit for \eqref{eq:center-v-L-plus-B}.
The argument has three distinct parts. First, we verify the spectral and
resolvent hypotheses of the Hamiltonian center-manifold theorem at the critical
shear. Second, we compute the cubic Hamiltonian coefficients in Darboux
coordinates. Third, we use the long-wave scaling dictated by the double root of
the dispersion relation to obtain a \(C^1\)-small perturbation of the stationary
KdV homoclinic orbit.

We first verify the hypotheses of Theorem~\ref{thm:Mielke_center_Darboux}. The
smoothness hypothesis (\textbf{H3}) holds with
\(\epsilon=\varepsilon_1\in\mathbb R\), with
\(\mathsf E=(-\delta_*,\delta_*)\) for \(\delta_*>0\) sufficiently small, and
with \(\mathsf U=U_*\cap\operatorname{dom}(\mathbf L)\) for \(U_*\) a
sufficiently small neighborhood of \(0\). Hypothesis (\textbf{H4}) follows from
the construction of \(\tilde H^{\varepsilon_1}\) and
\(\tilde\Omega^{\varepsilon_1}\), while (\textbf{H5}) follows from
\eqref{eq:center-symH2}--\eqref{eq:center-symvH2} with reverser
\(\mathsf S=S\). It remains to verify the spectral hypothesis \((\textbf{H1})\)
and the resolvent hypothesis \((\textbf{H2})\).

\subsection{Critical shear spectrum}

By Lemma~\ref{lembfL}, \(\mathbf L\) is a closed densely defined operator on
\(M_0\), and its graph norm is equivalent to the \(X^2\)-norm on
\(\operatorname{dom}(\mathbf L)\). Since \(X^2\hookrightarrow X^1\) compactly,
the resolvent of \(\mathbf L\), once nonempty, is compact. Thus the spectrum is
discrete, consisting of isolated eigenvalues of finite algebraic multiplicity.
We now identify the part of the spectrum on the imaginary axis.

\begin{lemma}\label{lemnece}
Assume \(\omega_0=1\) and \(\sigma_0>1/3\). A complex number \(\lambda\) is an
eigenvalue of \(\mathbf L\) if and only if
\begin{equation}\label{dispersion relation omega sigma}
        \lambda\cos\lambda=(1-\sigma_0\lambda^2)\sin\lambda .
\end{equation}
In particular, the only purely imaginary eigenvalue of \(\mathbf L\) is
\(\lambda=0\).
\end{lemma}

\begin{proof}
Let \(u=(\Phi,\theta,Z,\eta)\in\operatorname{dom}(\mathbf L)\setminus\{0\}\)
satisfy \(\mathbf L u=\lambda u\). By \eqref{eq:center-L-formula}, this is
equivalent to
\begin{align}
\theta_y+\frac{1}{2\sigma_0}\bigl(\theta_y(1)+\eta\bigr)
\left(\frac13-y^2\right)+\eta(2y-1)
&=\lambda\Phi, \label{eq:E1}\\
-\Phi_y+3(\Phi(1)+Z)(1-y)y
&=\lambda\theta, \label{eq:E2}\\
-\theta_y(1)-\eta
&=\lambda Z, \label{eq:E3}\\
3(\Phi(1)+Z)
&=\lambda\eta. \label{eq:E4}
\end{align}

If \(\lambda=0\), then \eqref{eq:E4} gives \(\Phi(1)+Z=0\), while
\eqref{eq:E2} gives \(\Phi_y=0\). Since \(\int_0^1\Phi\,dy=0\), we have
\(\Phi\equiv0\) and \(Z=0\). Equations \eqref{eq:E1} and \eqref{eq:E3} reduce
to \(\theta_y+\eta(2y-1)=0\) and \(\theta_y(1)+\eta=0\). Together with
\(\theta(0)=\theta(1)=0\), this gives \(\theta(y)=\eta(y-y^2)\). Hence
\(\lambda=0\) is an eigenvalue, with eigenvector
\begin{equation}\label{eq:zero-eigenvector}
        (\Phi,\theta,Z,\eta)=\bigl(0,(1-y)y,0,1\bigr).
\end{equation}

We now assume \(\lambda\neq0\). Equations \eqref{eq:E4} and \eqref{eq:E2} give
\begin{align}
        \Phi(1)+Z=\frac{\lambda}{3}\eta,\qquad \theta=-\frac{1}{\lambda}\Phi_y+\eta(1-y)y.
        \label{eq:theta_in_terms_of_Phi}
\end{align}
Therefore
\begin{equation}\label{eq:theta_y_in_terms_of_Phi}
        \theta_y=-\frac{1}{\lambda}\Phi_{yy}+\eta(1-2y),
        \qquad
        \theta_y(1)=-\frac{1}{\lambda}\Phi_{yy}(1)-\eta .
\end{equation}
Equation \eqref{eq:E3} gives \(Z=\Phi_{yy}(1)/\lambda^2\), and hence
\begin{equation}\label{eq:eta_from_Phi}
        \eta
        =
        \frac{3}{\lambda}
        \left(\Phi(1)+\frac{\Phi_{yy}(1)}{\lambda^2}\right).
\end{equation}
Substituting \eqref{eq:theta_y_in_terms_of_Phi} into \eqref{eq:E1} yields
\begin{equation}\label{eq:Phi_forced_ODE}
        \Phi_{yy}+\lambda^2\Phi
        =
        -\frac{\Phi_{yy}(1)}{2\sigma_0}
        \left(\frac13-y^2\right)
        \quad\text{in }L^2(0,1).
\end{equation}

Set \(A:=\Phi_{yy}(1)\). A particular solution of
\eqref{eq:Phi_forced_ODE} is
\begin{equation}\label{eq:Phi_particular}
        \Phi_p(y)
        =
        \frac{A}{2\sigma_0\lambda^2}
        \left(y^2-\frac13\right)
        -\frac{A}{\sigma_0\lambda^4}.
\end{equation}
Thus
\begin{equation}\label{eq:Phi_general}
        \Phi(y)
        =
        c_1\cos(\lambda y)+c_2\sin(\lambda y)
        +\frac{A}{2\sigma_0\lambda^2}
        \left(y^2-\frac13\right)
        -\frac{A}{\sigma_0\lambda^4}.
\end{equation}
The boundary conditions \(\Phi_y(0)=\Phi_y(1)=0\) imply
\begin{equation}\label{eq:c1_from_A}
        c_2=0,\qquad
        c_1=\frac{A}{\sigma_0\lambda^3\sin\lambda}.
\end{equation}
For a nontrivial eigenvector, \(\sin\lambda\neq0\). Indeed, if
\(\sin\lambda=0\), then the boundary conditions give \(A=0\); since
\(A=\Phi_{yy}(1)\), the formula \eqref{eq:Phi_general} then forces
\(c_1=c_2=0\), and hence \(u=0\).

Using \eqref{eq:c1_from_A}, we compute
\begin{equation}\label{eq:Phi_yy_general}
        \Phi_{yy}(y)
        =
        -\frac{A\cos(\lambda y)}{\sigma_0\lambda\sin\lambda}
        +\frac{A}{\sigma_0\lambda^2}.
\end{equation}
Evaluating at \(y=1\) and using \(A=\Phi_{yy}(1)\), we obtain
\begin{equation}\label{eq:trace_compatibility}
        A
        =
        -\frac{A}{\sigma_0\lambda}\cot\lambda
        +\frac{A}{\sigma_0\lambda^2}.
\end{equation}
Since \(A\neq0\), \eqref{eq:trace_compatibility} is equivalent to
\eqref{dispersion relation omega sigma}. This proves the necessity.

Conversely, suppose that \(\lambda\neq0\) satisfies
\eqref{dispersion relation omega sigma}. Then \(\sin\lambda\neq0\). Choose
\(A\neq0\), define \(c_1,c_2\) by \eqref{eq:c1_from_A}, and define \(\Phi\) by
\eqref{eq:Phi_general}. The relation \eqref{eq:trace_compatibility} follows
from \eqref{dispersion relation omega sigma}; hence \(A=\Phi_{yy}(1)\) and
\(\Phi_y(0)=\Phi_y(1)=0\). Also
\begin{equation}\label{eq:Phi_mean_zero}
        \int_0^1\Phi(y)\,dy
        =
        c_1\frac{\sin\lambda}{\lambda}
        -\frac{A}{\sigma_0\lambda^4}
        =
        0.
\end{equation}
Define \(Z=\Phi_{yy}(1)/\lambda^2\), define \(\eta\) by
\eqref{eq:eta_from_Phi}, and define \(\theta\) by
\eqref{eq:theta_in_terms_of_Phi}. Then
\(u=(\Phi,\theta,Z,\eta)\in\operatorname{dom}(\mathbf L)\), and direct
substitution verifies \eqref{eq:E1}--\eqref{eq:E4}. Hence \(\lambda\) is an
eigenvalue.

Finally, take \(\lambda=i\alpha\), \(\alpha\in\mathbb R\). Substitution into
\eqref{dispersion relation omega sigma} gives
\begin{equation}\label{eq:imaginary_eigen_equation}
        \alpha\cosh\alpha=(1+\sigma_0\alpha^2)\sinh\alpha .
\end{equation}
If \(\alpha=0\), this is the eigenvalue already found. If \(\alpha\neq0\), then
\begin{equation}\label{eq:imaginary_coth_equation}
        \alpha\coth\alpha=1+\sigma_0\alpha^2 .
\end{equation}
But \(\alpha\coth\alpha<1+\alpha^2/3\) for \(\alpha\neq0\), contradicting
\(\sigma_0>1/3\). Thus no nonzero purely imaginary eigenvalue exists.
\end{proof}

By Kato's perturbation theory \cite[Chapter~III, Theorem~6.17]{kato2013perturbation},
the isolated spectral point \(0\) has an associated spectral projection. Let
\(\Gamma\) be a contour in the resolvent set of \(\mathbf L\), enclosing \(0\)
and no other point of \(\sigma(\mathbf L)\). Define
\begin{equation}\label{kato map}
        \mathbf P_1 v
        :=
        \frac{1}{2\pi i}
        \oint_\Gamma
        (\lambda\mathbf I-\mathbf L)^{-1}v\,d\lambda,
        \qquad
        \mathbf P_2:=\mathbf I-\mathbf P_1.
\end{equation}
Then \begin{equation}\label{eq:M0-spectral-decomposition} M_0= M_1\oplus M_2, \qquad M_1:=\mathbf P_1 M_0,\qquad M_2:=\mathbf P_2 M_0 . \end{equation} The space \(M_1\) is finite dimensional. Since \(\mathbf P_1\) is a spectral projection, both \(M_1\) and \(M_2\) are invariant under \(\mathbf L\). Moreover, Lemma~\ref{lemnece} shows that the spectrum of \(\mathbf L|_{M_1}\) is contained in the imaginary axis, whereas \[ \mathbf L_2 := \mathbf L|_{\operatorname{dom}(\mathbf L)\cap M_2} \] has no spectrum on the imaginary axis. Hence hypothesis \textbf{(H1)} holds.

\subsection{Resolvent estimate on the hyperbolic complement}

We now verify the resolvent hypothesis \textup{(H2)}. Similar estimates appear
in \cite[Lemma~2.1]{iooss1992water},
\cite[Proposition~3.2]{buffoni1996plethora},
\cite[Proposition~2]{buffoni1999multiplicity},
\cite[Lemma~3.4]{groves2001spatial},
\cite[Lemma~3.4]{groves2007spatial}, and
\cite[Lemma~3.8]{hur2023unstable}.

\begin{lemma}[Resolvent estimate]\label{resolvent estimate}
Assume \(\omega_0=1\) and \(\sigma_0>1/3\). There exist constants
\(\alpha_0>0\) and \(C>0\) such that \(i\alpha\in\rho(\mathbf L)\) and
\begin{equation}\label{ineq 4.3}
        \|(\mathbf L-i\alpha I)^{-1}\|_{M_0\to M_0}
        \le \frac{C}{|\alpha|},
        \qquad |\alpha|>\alpha_0 .
\end{equation}
\end{lemma}

\begin{proof}
By Lemma~\ref{lemnece}, \(i\alpha\in\rho(\mathbf L)\) for every
\(\alpha\neq0\). Let
\(v=(\Phi,\theta,Z,\eta)\in\operatorname{dom}(\mathbf L)\) and
\(\tilde v=(\tilde\Phi,\tilde\theta,\tilde Z,\tilde\eta)\in M_0\) satisfy
\((\mathbf L-i\alpha I)v=\tilde v\). By \eqref{eq:center-L-formula},
\begin{equation}\label{L-isigma=v}
\begin{aligned}
\theta_y+\frac{1}{2\sigma_0}(\theta_y(1)+\eta)
\left(\frac13-y^2\right)+\eta(2y-1)-i\alpha\Phi
&=\tilde\Phi,\\
-\Phi_y+3(\Phi(1)+Z)(1-y)y-i\alpha\theta
&=\tilde\theta,\\
-\theta_y(1)-\eta-i\alpha Z
&=\tilde Z,\\
3(\Phi(1)+Z)-i\alpha\eta
&=\tilde\eta .
\end{aligned}
\end{equation}
In this proof \(\|\cdot\|\) denotes the \(L^2(0,1)\)-norm, and
\(A\lesssim B\) means \(A\le CB\), with \(C\) independent of \(\alpha\).

Taking the \(L^2\)-norms of the first two equations in
\eqref{L-isigma=v}, squaring and adding, gives
\begin{equation}\label{first}
\begin{aligned}
\|\theta_y\|^2+\|\Phi_y\|^2+\alpha^2(\|\theta\|^2+\|\Phi\|^2)
&=
\|\tilde\theta-3(\Phi(1)+Z)(y-y^2)\|^2 \\
&\quad+
\left\|
\tilde\Phi-\frac{\theta_y(1)+\eta}{2\sigma_0}
\left(\frac13-y^2\right)-\eta(2y-1)
\right\|^2 .
\end{aligned}
\end{equation}
Differentiating the first two equations in \eqref{L-isigma=v} with respect to
\(y\), and using \(\Phi_y(0)=\Phi_y(1)=0\), yields
\begin{equation}\label{second}
\begin{aligned}
\|\theta_{yy}\|^2+\|\Phi_{yy}\|^2+\alpha^2(\|\theta_y\|^2+\|\Phi_y\|^2)
&=
\|\tilde\theta_y-3(\Phi(1)+Z)(1-2y)\|^2 \\
&\quad+
\left\|
\tilde\Phi_y+\frac{\theta_y(1)+\eta}{\sigma_0}y-2\eta
\right\|^2 .
\end{aligned}
\end{equation}
The last two equations in \eqref{L-isigma=v} imply
\begin{align}
        \alpha^2|Z|^2
        &\lesssim |\tilde Z|^2+|\theta_y(1)|^2+|\eta|^2,
        \label{ineq 4.10}\\
        \alpha^2|\eta|^2
        &\lesssim |\tilde\eta|^2+|\Phi(1)|^2+|Z|^2 .
        \label{inequality 410}
\end{align}
Combining \eqref{first}--\eqref{inequality 410}, and using the trace estimate
\eqref{traceb}, gives
\begin{equation}\label{inequality 49}
\begin{aligned}
&\alpha^2
\left(
\|\theta\|^2+\|\Phi\|^2+\|\theta_y\|^2+\|\Phi_y\|^2+|Z|^2+|\eta|^2
\right) \\
&\quad\lesssim
\|\tilde\theta\|_{H^1}^2+\|\tilde\Phi\|_{H^1}^2
+|\tilde Z|^2+|\tilde\eta|^2
+\|\Phi\|_{H^1}^2+|Z|^2+|\eta|^2+|\theta_y(1)|^2 .
\end{aligned}
\end{equation}

It remains to control the trace \(\theta_y(1)\). Differentiating the first
equation in \eqref{L-isigma=v}, testing against \(y^n\), and integrating over
\((0,1)\), we obtain
\begin{equation}\label{eq:theta-trace-resolvent}
\begin{aligned}
\left(1-\frac{1}{\sigma_0(n+2)}\right)\theta_y(1)
&=
n\int_0^1 y^{n-1}\theta_y\,dy
+\int_0^1 y^n\tilde\Phi_y\,dy
+i\alpha\int_0^1 y^n\Phi_y\,dy  \\
&\quad
+\frac{\eta}{\sigma_0(n+2)}
-\frac{2\eta}{n+1}.
\end{aligned}
\end{equation}
For \(n\) sufficiently large, depending only on \(\sigma_0\), H\"older's
inequality gives
\begin{equation}\label{third}
        |\theta_y(1)|^2
        \lesssim
        n\|\theta_y\|^2
        +\frac{1}{n}\|\tilde\Phi_y\|^2
        +\frac{\alpha^2}{n}\|\Phi_y\|^2
        +\frac{1}{n^2}|\eta|^2 .
\end{equation}
Substituting \eqref{third} into \eqref{inequality 49} gives
\begin{equation}\label{inequality 499}
\begin{aligned}
&\alpha^2
\left(
\|\theta\|^2+\|\Phi\|^2+\|\theta_y\|^2+\|\Phi_y\|^2+|Z|^2+|\eta|^2
\right) \\
&\quad\lesssim
\|\tilde\theta\|_{H^1}^2+\|\tilde\Phi\|_{H^1}^2
+|\tilde Z|^2+|\tilde\eta|^2
+\|\Phi\|_{H^1}^2+|Z|^2+|\eta|^2
+n\|\theta_y\|^2+\frac{\alpha^2}{n}\|\Phi_y\|^2 .
\end{aligned}
\end{equation}
Choose \(\alpha_0\) large and then, for \(|\alpha|>\alpha_0\), choose
\(n\in\mathbb N\) with \(n\le |\alpha|<n+1\). Absorbing the lower-order terms in
\eqref{inequality 499}, we find
\begin{equation}\label{inequality 412}
        \|v\|_{M_0}
        \lesssim
        \frac{1}{|\alpha|}\|\tilde v\|_{M_0},
        \qquad |\alpha|>\alpha_0 .
\end{equation}
This is exactly \eqref{ineq 4.3}.
\end{proof}

\subsection{Cubic Hamiltonian coefficients and the reduced system}

We now apply Theorem~\ref{thm:Mielke_center_Darboux}. The center space is
two-dimensional, and we first choose a symplectic basis.

\begin{lemma}[Center space]\label{lem:center-space}
Assume \(\omega_0=1\) and \(\sigma_0>1/3\). Then \(0\) is an eigenvalue of
\(\mathbf L\) with algebraic multiplicity two and geometric multiplicity one.
Moreover,
\begin{equation}\label{eq:L-kernel-chain}
        \ker(\mathbf L^2)=\ker(\mathbf L^3)
        =
        \operatorname{span}\{\mathbf\Phi_1,\mathbf\Phi_2\},
\end{equation}
where
\begin{equation}\label{eq:Phi-chain}
        \mathbf\Phi_1=\vet{0}{(1-y)y}{0}{1},
        \qquad
        \mathbf\Phi_2=\vet{0}{0}{1/3}{0},
        \qquad
        \mathbf L\mathbf\Phi_1=0,\quad
        \mathbf L\mathbf\Phi_2=\mathbf\Phi_1 .
\end{equation}
\end{lemma}

\begin{proof}
By \eqref{eq:center-L-formula}, \(\ker\mathbf L=\operatorname{span}\{\mathbf\Phi_1\}\).
Solving \(\mathbf L v=\mathbf\Phi_1\) gives
\(v=\mathbf\Phi_2+c\mathbf\Phi_1\), and the equation
\(\mathbf L v=\mathbf\Phi_2\) has no solution. Hence the Jordan chain stops at
length two, and \eqref{eq:L-kernel-chain} follows.
\end{proof}

By \eqref{eq:center-Omega-tilde}, the constant symplectic form
\(\Psi=\tilde\Omega^0_0\) is
\begin{equation}\label{PSI}
\begin{aligned}
\Psi[u^1,u^2]
&=
\left[
\left(3\sigma_0-1+\frac{\omega_0}{4}\right)(\Phi^2(1)+Z^2)
+\Phi^2(1)-2\omega_0\int_0^1y\Phi^2\,dy
\right]\eta^1 \\
&\quad
-\left[
\left(3\sigma_0-1+\frac{\omega_0}{4}\right)(\Phi^1(1)+Z^1)
+\Phi^1(1)-2\omega_0\int_0^1y\Phi^1\,dy
\right]\eta^2 \\
&\quad
+\int_0^1\theta^2_y
\left(\Phi^1-\frac32\left(y^2-\frac13\right)(\Phi^1(1)+Z^1)\right)\,dy \\
&\quad
-\int_0^1\theta^1_y
\left(\Phi^2-\frac32\left(y^2-\frac13\right)(\Phi^2(1)+Z^2)\right)\,dy,
\end{aligned}
\end{equation}
where \(u^j=(\Phi^j,\theta^j,Z^j,\eta^j)\). In particular,
\begin{equation}\label{eq:Psi-Phi1-Phi2}
        \Psi[\mathbf\Phi_1,\mathbf\Phi_2]
        =
        \sigma_0-\frac13>0 .
\end{equation}
Thus
\begin{equation}\label{e f}
        \mathsf e
        =
        \left(\sigma_0-\frac13\right)^{-1/2}\mathbf\Phi_1,
        \qquad
        \mathsf f
        =
        \left(\sigma_0-\frac13\right)^{-1/2}\mathbf\Phi_2
\end{equation}
is a symplectic basis of \(M_1=\mathbf P_1M_0\). We identify \(M_1\) with
\(\mathbb R^2\) by
\begin{equation}\label{M_1_identification}
        (q,p)\longmapsto q\mathsf e+p\mathsf f.
\end{equation}
In these coordinates, the reduced symplectic form is canonical:
\begin{equation}\label{eq:canonical-Upsilon}
        \Upsilon[v^1,v^2]=q^1p^2-p^1q^2,
        \qquad
        v^j=(q^j,p^j).
\end{equation}

Let \(\tilde r\) be the Darboux-center-manifold graph from
Theorem~\ref{thm:Mielke_center_Darboux}. The reduced Hamiltonian is
\begin{equation}\label{reduced Hamiltonian}
        \bar H^{\varepsilon_1}(q,p)
        =
        \hat H^{\varepsilon_1}
        \bigl(q\mathsf e+p\mathsf f+\tilde r(q,p,\varepsilon_1)\bigr).
\end{equation}
The reverser acts by \(S(q,p)=(q,-p)\). Since the origin is an equilibrium for
all small \(\varepsilon_1\), we have \(\tilde r(0,\varepsilon_1)=0\). Together
with \(d_1\tilde r(0,0)=0\), this gives
\begin{align}
        \|\tilde r(q,p,\varepsilon_1)\|
        &=O(|(q,p)|\,|(q,p,\varepsilon_1)|), \label{estr}\\
        \|d_1\tilde r(q,p,\varepsilon_1)\|
        &=O(|(q,p,\varepsilon_1)|). \label{eq:dr-estimate}
\end{align}
Moreover, by \eqref{eq:center-Htilde-def} and
\(\hat H^{\varepsilon_1}=\tilde H^{\varepsilon_1}\circ\jmath\),
\(\hat H^{\varepsilon_1}\) is polynomial in \(\varepsilon_1\), and we write
\begin{equation}\label{expansion of hat H}
        \hat H^{\varepsilon_1}
        =
        \hat H^0+\varepsilon_1\hat H^1+\varepsilon_1^2\hat H^2 .
\end{equation}

The following coefficient identities are the point at which the KdV normal form
enters the argument. They identify the quadratic kinetic term, the cubic
nonlinearity, and the parameter term generated by the critical shear.

\begin{lemma}[Cubic Hamiltonian coefficients]\label{Lem coff}
For \(v'' = (q'',p'') = q'' \mathsf{e} + p'' \mathsf{f}\),
\(v' = (q',p') = q' \mathsf{e} + p' \mathsf{f}\), and
\(v = (q,p) = q \mathsf{e} + p \mathsf{f}\), we have
\begin{align}
    \label{d2H00}
    d^2 \hat{H}^{0}_0[v',v] &= p' p, \\
    \label{d3H00}
    d^3 \hat{H}^{0}_0[v'',v', v]
    &= - \left( \sigma_0 - \frac{1}{3}\right)^{-\frac{3}{2}} q'' q' q
    - \frac{3}{2}\left( \sigma_0 - \frac{1}{3}\right)^{-\frac{1}{2}}
    \left(  q'' p' p + p'' p'q + p'' q' p \right), \\
    \label{d2H10}
    d^2 \hat{H}^{1}_0[v',v]
    &= -\left( \sigma_0 - \frac{1}{3}\right)^{-1} q'q.
\end{align}
\end{lemma}

Differentiating \eqref{reduced Hamiltonian}, using
\eqref{estr}--\eqref{eq:dr-estimate}, Lemma~\ref{lemequ}, and the coefficient
computations in Lemmas~\ref{lemma A1} and~\ref{Lem coff}, gives
\begin{align}
d\bar H^{\varepsilon_1}(q,p)[\mathsf e]
&=
-\varepsilon_1\left(\sigma_0-\frac13\right)^{-1}q
-\frac12\left(\sigma_0-\frac13\right)^{-3/2}q^2
\nonumber\\
&\quad
+O(|p|\,|(q,p,\varepsilon_1)|)
+O(|(q,p)|\,|(q,p,\varepsilon_1)|^2),
\label{eq:dH-e}\\
d\bar H^{\varepsilon_1}(q,p)[\mathsf f]
&=
p+O(|(q,p)|\,|(q,p,\varepsilon_1)|).
\label{eq:dH-f}
\end{align}
Therefore Hamilton's equations for
\((\mathcal M,\Upsilon,\bar H^{\varepsilon_1})\) take the form
\begin{equation}\label{q_x p_x}
\begin{aligned}
\dot q
&=
p+\mathcal R_1(q,p,\varepsilon_1),\\
\dot p
&=
\left(\sigma_0-\frac13\right)^{-1}\varepsilon_1 q
+\frac12\left(\sigma_0-\frac13\right)^{-3/2}q^2
+\mathcal R_2(q,p,\varepsilon_1),
\end{aligned}
\end{equation}
where reversibility gives \(\mathcal R_1\) odd and \(\mathcal R_2\) even in
\(p\), and
\begin{align}
        \mathcal R_1
        &=
        O(|(q,p)|\,|(q,p,\varepsilon_1)|), \label{mathcal R1}\\
        \mathcal R_2
        &=
        O(|p|\,|(q,p,\varepsilon_1)|)
        +O(|(q,p)|\,|(q,p,\varepsilon_1)|^2). \label{mathcal R2}
\end{align}

Introduce the long-wave variables
\begin{equation}\label{eq:QPX-scaling-main}
        Q(\bar x)
        =
        \frac{q(x)}{\varepsilon_1\sqrt{\sigma_0-\frac13}},
        \qquad
        P(\bar x)
        =
        \frac{p(x)}{\varepsilon_1^{3/2}},
        \qquad
        \bar x
        =
        \sqrt{\varepsilon_1}\left(\sigma_0-\frac13\right)^{-1/2}x .
\end{equation}
Then \eqref{q_x p_x} becomes
\begin{equation}\label{Q P varepsilon}
\begin{aligned}
Q_{\bar x}
&=
P+
\varepsilon_1^{-3/2}
\mathcal R_1\!\left(
\varepsilon_1\sqrt{\sigma_0-\frac13}\,Q,
\varepsilon_1^{3/2}P,\varepsilon_1
\right),\\
P_{\bar x}
&=
Q+\frac12Q^2+
\varepsilon_1^{-2}\sqrt{\sigma_0-\frac13}\,
\mathcal R_2\!\left(
\varepsilon_1\sqrt{\sigma_0-\frac13}\,Q,
\varepsilon_1^{3/2}P,\varepsilon_1
\right).
\end{aligned}
\end{equation}
Set \(\delta=\varepsilon_1^{1/4}\), and define
\begin{align}
\overline{\mathcal R}_1(Q,P,\delta)
&:=
\delta^{-6}
\mathcal R_1\!\left(
\delta^4\sqrt{\sigma_0-\frac13}\,Q,\delta^6P,\delta^4
\right),
\label{bar R1}\\
\overline{\mathcal R}_2(Q,P,\delta)
&:=
\delta^{-8}
\mathcal R_2\!\left(
\delta^4\sqrt{\sigma_0-\frac13}\,Q,\delta^6P,\delta^4
\right).
\label{bar R2}
\end{align}
The rescaled system is
\begin{equation}\label{Q P varepsilon 2}
\begin{aligned}
Q_{\bar x}
&=
P+\overline{\mathcal R}_1(Q,P,\delta),\\
P_{\bar x}
&=
Q+\frac12Q^2
+\sqrt{\sigma_0-\frac13}\,\overline{\mathcal R}_2(Q,P,\delta).
\end{aligned}
\end{equation}

\begin{lemma}\label{R1 and R2 are C1}
The functions \(\overline{\mathcal R}_1\) and \(\overline{\mathcal R}_2\) are
\(C^1\) in \((Q,P,\delta)\) for \(\delta\) sufficiently small. Moreover, after
defining their values at \(\delta=0\) to be zero, one has
\begin{equation}\label{eq:barR-C1-small}
        \overline{\mathcal R}_j(Q,P,0)=0,
        \qquad
        d\overline{\mathcal R}_j(Q,P,0)=0,
        \qquad j=1,2,
\end{equation}
locally uniformly for \((Q,P)\) in bounded sets.
\end{lemma}

\begin{proof}
For \(\delta\neq0\), this follows from the smoothness of
\(\mathcal R_1,\mathcal R_2\). The estimates
\eqref{mathcal R1}--\eqref{mathcal R2} and Lemma~\ref{est deri} imply that
\(\overline{\mathcal R}_j(Q,P,\delta)\to0\) and
\(d\overline{\mathcal R}_j(Q,P,\delta)\to0\) as \(\delta\to0\), for
\(j=1,2\). Defining the values at \(\delta=0\) to be zero gives the claimed
\(C^1\)-regularity and \eqref{eq:barR-C1-small}.
\end{proof}

Thus \eqref{Q P varepsilon 2} is a \(C^1\)-small reversible perturbation of the
limiting KdV system as \(\delta\to0\). This is the only perturbative input
needed below to persist the reversible homoclinic orbit.

At \(\delta=0\), \eqref{Q P varepsilon 2} reduces to
\begin{equation}\label{reduced_ode}
        Q_{\bar x}=P,\qquad
        P_{\bar x}=Q+\frac12Q^2 .
\end{equation}
Equivalently,
\begin{equation}\label{eq:stationary-KdV}
        Q''=Q+\frac12Q^2 .
\end{equation}
The Hamiltonian
\begin{equation}\label{eq:KdV-Hamiltonian}
        \mathscr H(Q,P)
        =
        \frac12P^2-\frac12Q^2-\frac16Q^3
\end{equation}
is conserved. On the zero-energy level,
\(P^2=Q^2(1+Q/3)\), and the homoclinic orbit to \((0,0)\) is
\begin{equation}\label{eq:KdV-homoclinic}
        Q(\bar x)=-3\sech^2\left(\frac{\bar x}{2}\right),
        \qquad
        P(\bar x)=3\sech^2\left(\frac{\bar x}{2}\right)
        \tanh\left(\frac{\bar x}{2}\right).
\end{equation}

By Lemma~\ref{R1 and R2 are C1}, the local stable and unstable manifolds of
\((0,0)\) for \eqref{Q P varepsilon 2} vary continuously in \(\delta\); see
\cite{Yi1993}. Hence, for \(\delta>0\) sufficiently small, the stable manifold
has a branch close to the homoclinic branch in \eqref{eq:KdV-homoclinic}. This
branch intersects the fixed-point set \(\{P=0\}\) of the reverser. Translating
in \(\bar x\), let the intersection occur at \(\bar x=0\). By reversibility,
\((Q^\delta(-\bar x),-P^\delta(-\bar x))\) is a solution with the same initial
condition at \(\bar x=0\); uniqueness gives
\(Q^\delta(\bar x)=Q^\delta(-\bar x)\) and
\(P^\delta(\bar x)=-P^\delta(-\bar x)\). Since the solution tends to \(0\) as
\(\bar x\to+\infty\), it also tends to \(0\) as \(\bar x\to-\infty\). Thus
\eqref{Q P varepsilon 2} has a reversible homoclinic orbit for all sufficiently
small \(\delta>0\).

\subsection{Homoclinic persistence and physical reconstruction}

The homoclinic orbit constructed above gives, through
Theorem~\ref{thm:Mielke_center_Darboux}, a homoclinic orbit of
\eqref{eq:center-v-L-plus-B}. We then lift this orbit through the Darboux
coordinates, the fixed-domain transformation, and finally the reconstruction
result of Theorem~\ref{thm:Hamiltonian_equivalence}. This produces a solitary
traveling wave of the original free-boundary Euler system, rather than merely a
solution of the reduced normal form. It remains to record the leading-order
profile in physical variables.

Let \(\xi_{\mathrm{phys}}=x-ct\) be the physical traveling coordinate and let
\(d>0\) be the depth. We nondimensionalize by
\begin{equation}\label{eq:nondim-vars}
        \xi=\frac{\xi_{\mathrm{phys}}}{d},
        \qquad
        y=\frac{y_{\mathrm{phys}}}{d},
        \qquad
        \eta=\frac{\eta_{\mathrm{phys}}}{d},
\end{equation}
and
\begin{equation}\label{eq:nondim-params}
        \omega=\frac{\omega_{\mathrm{phys}}d}{c},
        \qquad
        \sigma=\frac{\sigma_{\mathrm{phys}}}{c^2d}.
\end{equation}
Thus
\begin{equation}\label{eq:eta-phys-from-dimless}
        \eta_{\mathrm{phys}}(\xi_{\mathrm{phys}})
        =
        d\,\eta\!\left(\frac{\xi_{\mathrm{phys}}}{d}\right).
\end{equation}
The small parameter is
\begin{equation}\label{eq:eps-def}
        \frac{\omega_{\mathrm{phys}}d}{c}=1+\varepsilon,
\end{equation}
and the capillary regime is
\begin{equation}\label{eq:sigma-regime}
        \frac{\sigma_{\mathrm{phys}}}{c^2d}>\frac13.
\end{equation}

In the normalized variables, \(M_1=\operatorname{span}\{\mathbf\Phi_1,\mathbf\Phi_2\}\),
with \(\mathbf\Phi_1,\mathbf\Phi_2\) given in \eqref{eq:Phi-chain}. Let
\(\Pi_\eta(\Phi,\theta,Z,\eta)=\eta\). Then
\begin{equation}\label{eq:eta-projection}
        \Pi_\eta(\mathbf\Phi_1)=1,
        \qquad
        \Pi_\eta(\mathbf\Phi_2)=0.
\end{equation}
Using the symplectic basis \eqref{e f}, we have
\begin{equation}\label{eq:Pi-eta-of-e-f}
        \Pi_\eta(\mathsf e)=\left(\sigma_0-\frac13\right)^{-1/2},
        \qquad
        \Pi_\eta(\mathsf f)=0.
\end{equation}

On the center manifold, the full state is
\begin{equation}\label{eq:center-graph}
        m(x)=q(x)\mathsf e+p(x)\mathsf f
        +\tilde r(q(x),p(x),\varepsilon_1).
\end{equation}
By \eqref{estr},
\begin{equation}\label{eq:r-estimate}
        \|\tilde r(q,p,\varepsilon_1)\|
        =
        O(|(q,p)|\,|(q,p,\varepsilon_1)|).
\end{equation}
Projecting \eqref{eq:center-graph} onto the surface component gives
\begin{equation}\label{eq:eta-q-leading}
        \eta(x)
        =
        \frac{q(x)}{\sqrt{\sigma_0-\frac13}}
        +\Pi_\eta\tilde r(q(x),p(x),\varepsilon_1).
\end{equation}

The long-wave scaling \eqref{eq:QPX-scaling-main} gives
\begin{equation}\label{eq:qPiEtae}
        q(x)=\varepsilon_1\sqrt{\sigma_0-\frac13}\,Q(\bar x),
        \qquad
        p(x)=\varepsilon_1^{3/2}P(\bar x).
\end{equation}
Hence \(|(q,p)|=O(\varepsilon_1)\) along the homoclinic orbit, and
\eqref{eq:r-estimate} gives
\begin{equation}\label{eq:r-order}
        \Pi_\eta\tilde r(q(x),p(x),\varepsilon_1)
        =
        O(\varepsilon_1^2).
\end{equation}
Combining \eqref{eq:eta-q-leading}--\eqref{eq:r-order}, we obtain the
nondimensional expansion
\begin{equation}\label{eq:eta-dimless-asympt}
        \eta(x)
        =
        \varepsilon_1 Q(\bar x)+O(\varepsilon_1^2),
        \qquad
        \bar x
        =
        \sqrt{\varepsilon_1}
        \left(\sigma_0-\frac13\right)^{-1/2}x.
\end{equation}
Identifying \(\varepsilon_1=\varepsilon\), using
\(\sigma_0=\sigma_{\mathrm{phys}}/(c^2d)\), and returning to physical
variables by \eqref{eq:eta-phys-from-dimless}, we obtain
\begin{equation}\label{eq:eta-phys-asympt}
\eta_{\mathrm{phys}}(\xi_{\mathrm{phys}})
=
d\,\varepsilon\,
Q\!\left(
\varepsilon^{1/2}
\left(\frac{\sigma_{\mathrm{phys}}}{c^2d}-\frac13\right)^{-1/2}
\frac{\xi_{\mathrm{phys}}}{d}
\right)
+O(d\varepsilon^2),
\end{equation}
where
\begin{equation}\label{eq:Q-final-profile}
        Q(X)=-3\sech^2(X/2).
\end{equation}
This is the expansion stated in Theorem~\ref{thm:main}.

\appendix
\section{Coefficient computations and remainder estimates}\label{sec:App1}

An application of Taylor expansion (together with \eqref{estr}) reveals the following lemma.
\begin{lemma}  \label{lemma A1}
As $(v,\varepsilon_1)\rightarrow (0,0)$, it holds that
\begin{equation} \label{A4 d1H}
    \begin{aligned}
        d\hat{H}^{\varepsilon_1}(v+\tilde{r}(v,\varepsilon_1))=\, & d^2 \hat{H}^0_0[v+\tilde{r}(v,\varepsilon_1)]+\frac{1}{2} d^3 \hat{H}^0_0[v+\tilde{r}(v,\varepsilon_1),v+\tilde{r}(v,\varepsilon_1)]\\
        +\,& \varepsilon_1 d^2\hat{H}^1_0[v+\tilde{r}(v,\varepsilon_1)]+O(|(q,p)| |(q,p,\varepsilon_1)|^2)\\
        =\,& d^2 \hat{H}^0_0[v+\tilde{r}(v,\varepsilon_1)]+\frac{1}{2} d^3 \hat{H}^0_0[v,v]  
        +\varepsilon_1 d^2\hat{H}^1_0[v] +  O(|(q,p)| |(q,p,\varepsilon_1)|^2)\\
        =\,& d^2\hat{H}^0_0[v] + O(|(q,p)| |(q,p,\varepsilon_1)|)
    \end{aligned}
\end{equation}
and
\begin{equation} \label{A3 d21H}
\begin{aligned}
    d^2 \hat{H}^{\varepsilon_1}(v+\tilde{r}(v,\varepsilon_1)) = \, &d^2 \hat{H}^0_0+ d^3 \hat{H}^0_0[v+\tilde{r}(v,\varepsilon_1)] + \varepsilon_1 d^2 \hat{H}^1_0 + O(|(q,p,\varepsilon_1)|^2).
\end{aligned}
\end{equation}

%    and
%\begin{equation} \label{A4 d1H1 d2H1}
%    \begin{aligned}
%        d_1 \hat{H}^1(v+\tilde{r}(v,\varepsilon_1))=d^2 \hat{H}^1(0)[v+\tilde{r}(v,\varepsilon_1)]+O(|(q,p)|^2), \, d_1 \hat{H}^2(v+\tilde{r}(v,\varepsilon_1))=O(|(q,p)|).
%    \end{aligned}
%\end{equation}

\end{lemma}

We now prove the cubic coefficient lemma stated in Section~\ref{sec:normalform}.

\begin{proof}[Proof of Lemma~\ref{Lem coff}]
    
Recalling \eqref{eq:center-Htilde-def}, by Taylor expansion, we obtain 
\begin{equation}
    \begin{aligned}
        &d^2 \hat{H}^{0}_0[(\Phi,\theta,Z,\eta),(\Phi,\theta,Z,\eta)]\\
        &\quad=\int_0^1 \theta_y^2 \; dy-\int_0^1\Phi_y^2 \; dy+6 \left(\Phi(1)+Z\right)\int_0^1 y \Phi_y \; dy \\
        &\qquad-3\left(\Phi(1)+Z\right)^2+\eta^2-4\omega_0\eta\int_0^1 \theta \;dy+\frac{\omega_0^2}{3}\eta^2-\omega_0\eta^2+9\sigma_0(\Phi(1) + Z)^2,
    \end{aligned}
\end{equation}

\begin{equation}
    \begin{aligned}
        &d^3 \hat{H}^{0}_0[(\Phi,\theta,Z,\eta),(\Phi,\theta,Z,\eta),(\Phi,\theta,Z,\eta)]\\
        &\quad=-3\eta\int_0^1 \theta_y^2 \;dy+3\eta\int_0^1 \Phi_y^2 \; dy-\omega_0^2\eta^3+3\omega_0\eta^3-3\eta^3+6\omega_0\eta^2\int_0^1 \theta \; dy\\
        &\qquad -18\left(\Phi(1)+Z\right)\int_0^1 y\Phi_y\theta_y \;dy +54\left(\Phi(1)+Z\right)\theta_{\mathrm{corr}}(1)\int_0^1 y\Phi_y \; dy\\
        &\qquad
        -18\eta\left(\Phi(1)+Z\right)\int_0^1 y\Phi_y \;dy+9\eta\left(\Phi(1)+Z\right)^2\\
        &\qquad+54\left(\Phi(1)+Z\right)^2 \int_0^1 y^2\theta_y \; dy-54\left(\Phi(1)+Z\right)^2\theta_{\mathrm{corr}}(1) 
         + 162 \sigma_0 (\Phi(1) + Z)^2 \theta_{\mathrm{corr}}(1),
    \end{aligned}
\end{equation}
and
\begin{align}
    d^2\hat{H}^{1}_0 [(\Phi,\theta,Z,\eta),(\Phi,\theta,Z,\eta)]=-4\eta \int_0^1 \theta \,dy+\frac{2\omega_0}{3}\eta^2-\eta^2.
\end{align}

Substituting $v=q\mathsf e+p\mathsf f$ into these multilinear forms and using
the symmetry of the derivatives gives the stated identities. The contractions
which enter the reduced equations are, in particular,
\begin{equation}
\begin{gathered}
        d^2\hat H^0_0[\mathsf e,\mathsf e]=0,
        \qquad
        d^2\hat H^0_0[\mathsf f,\mathsf f]=1,
        \qquad
        d^2\hat H^0_0[\mathsf e,\mathsf f]=0,\\
        d^3\hat H^0_0[\mathsf e,\mathsf e,\mathsf e]
        =-\left(\sigma_0-\frac13\right)^{-3/2},
        \qquad
        d^3\hat H^0_0[\mathsf e,\mathsf f,\mathsf f]
        =-\frac32\left(\sigma_0-\frac13\right)^{-1/2},\\
        d^2\hat H^1_0[\mathsf e,\mathsf e]
        =-\left(\sigma_0-\frac13\right)^{-1}.
\end{gathered}
\end{equation}
All remaining contractions either vanish or follow from symmetry. This proves
Lemma~\ref{Lem coff}.
\end{proof}

To prove Lemma \ref{R1 and R2 are C1} we need the following estimate.

\begin{lemma}\label{est deri}
    The following estimates hold for $\varepsilon_1 \ne 0$:
    \begin{equation} \label{A14 partial derivatives}
        \begin{aligned}
            \frac{\partial \mathcal{R}_1}{\partial q}(q,p,\varepsilon_1)&=O(|(q,p,\varepsilon_1)|) ,&
            \frac{\partial \mathcal{R}_1}{\partial p}(q,p,\varepsilon_1)&=O(|(q,p,\varepsilon_1)|),\\
            \frac{\partial \mathcal{R}_2}{\partial q}(q,p,\varepsilon_1)&=O(|p|)+O(|(q,\varepsilon_1)|^2) ,&
            \frac{\partial \mathcal{R}_2}{\partial p}(q,p,\varepsilon_1)&=O(|(q,p,\varepsilon_1)|), \\
            \frac{\partial \mathcal{R}_1}{\partial \varepsilon_1}(q,p,\varepsilon_1)&=O(|(q,p)|),  &
            \frac{\partial \mathcal{R}_2}{\partial \varepsilon_1}(q,p,\varepsilon_1)&=O(|p|)+O( |(q,\varepsilon_1)|^2).
        \end{aligned}
    \end{equation}

\end{lemma}

\begin{proof}
Recalling \eqref{e f}, \eqref{reduced Hamiltonian}, \eqref{q_x p_x}, and $(iv)$ in Theorem~\ref{thm:Mielke_center_Darboux} and writing 
\begin{align} \label{R1 A1}
    \mathcal{R}_1(q,p,\varepsilon_1)=d \bar H^{\varepsilon_1}(v)[\mathsf{f}]-p
  = d \hat{H}^{\varepsilon_1} (v + \tilde{r}( v, \varepsilon_1)\bigr)\bigr[\mathsf{f}+d_1 \tilde{r}(v,\varepsilon_1)[\mathsf{f}]\bigr]-p
\end{align}
and
\begin{equation} \label{R2 A2}
\begin{aligned}
    \mathcal{R}_2(q,p,\varepsilon_1)=\, &-d\bar H^{\varepsilon_1}(v)[\mathsf{e}]-\frac{1}{\sigma_0-\frac{1}{3}}\varepsilon_1 q
        - \frac{1}{2\left(\sigma_0-\frac{1}{3}\right)^{3/2}} q^2\\
  =\, & -d \hat{H}^{\varepsilon_1} (v + \tilde{r}( v, \varepsilon_1)\bigr)\bigr[\mathsf{e}+d_1 \tilde{r}(v,\varepsilon_1)[\mathsf{e}]\bigr]-\frac{1}{\sigma_0-\frac{1}{3}}\varepsilon_1 q
        - \frac{1}{2\left(\sigma_0-\frac{1}{3}\right)^{3/2}} q^2.
\end{aligned}
\end{equation}

We record the details for the two estimates in which the cancellations from
Lemma~\ref{Lem coff} are used most explicitly, namely
$\partial_q\mathcal R_2$ and $\partial_{\varepsilon_1}\mathcal R_2$. The
remaining estimates follow from the same expansion and the bounds
\eqref{estr}--\eqref{eq:dr-estimate}. A direct computation, together with
Lemma~\ref{lemma A1} and \eqref{estr}, gives
\begin{equation}
\begin{aligned}
\frac{\partial \mathcal{R}_2}{\partial q}(q,p,\varepsilon_1) = & 
- d^2 \hat{H}^0_0\bigr[\mathsf{e},\mathsf{e}\bigr] - 2 d^2 \hat{H}^0_0\bigr[\mathsf{e},d_1 \tilde{r}(v,\varepsilon_1)[\mathsf{e}]\bigr] 
-d^3 \hat{H}^0_0\bigr[v,\mathsf{e},\mathsf{e}\bigr] -  \varepsilon_1 d^2 \hat{H}^1_0\bigr[\mathsf{e},\mathsf{e}\bigr] \\- &\, d^2\hat{H}^0_0\bigr[v,d^2\tilde{r}(v,\varepsilon_1)[\mathsf{e},\mathsf{e}]\bigr] -\frac{1}{\sigma_0-\frac{1}{3}}\varepsilon_1 
        - \frac{1}{\left(\sigma_0-\frac{1}{3}\right)^{3/2}} q + O(|(q,p,\varepsilon_1)|^2).
\end{aligned}
\end{equation}
Then Lemma~\ref{lemequ} implies
\begin{equation}
\begin{aligned}
d^2 \hat{H}^0_0\bigr[\mathsf{e},d_1 \tilde{r}(v,\varepsilon_1)[\mathsf{e}]\bigr] = &\Psi \bigr[\mathbf{L}\mathsf{e},d_1 \tilde{r}(v,\varepsilon_1)[\mathsf{e}]\bigr] = 0, \\ d^2\hat{H}^0_0\bigr[v,d^2\tilde{r}(v,\varepsilon_1)[\mathsf{e},\mathsf{e}]\bigr] = &\Psi\bigr[p \mathsf{e}, d^2\tilde{r}(v,\varepsilon_1)[\mathsf{e},\mathsf{e}]\bigr] = O(|p|),
\end{aligned}
\end{equation}
and Lemma~\ref{Lem coff} gives
\begin{equation}
\begin{gathered}
d^2 \hat{H}^0_0\bigr[\mathsf{e},\mathsf{e}\bigr] = 0, \qquad d^3 \hat{H}^0_0\bigr[v,\mathsf{e},\mathsf{e}\bigr] = - \frac{1}{\left(\sigma_0-\frac{1}{3}\right)^{3/2}} q,  \qquad d^2 \hat{H}^1_0\bigr[\mathsf{e},\mathsf{e}\bigr]  =  -\frac{1}{\sigma_0-\frac{1}{3}}.
\end{gathered}
\end{equation}

Combining the preceding cancellations gives the asserted estimate for
\(\partial_q\mathcal R_2\). Similarly, using also
\[
d_2 \tilde{r}(v,\varepsilon_1) = d_2 \tilde{r}(v,\varepsilon_1) - d_2 \tilde{r}(0,\varepsilon_1) = O(|(q,p)|),
\]
we obtain
\begin{equation}
\begin{aligned}
\frac{\partial \mathcal{R}_2}{\partial \varepsilon_1 }(q,p,\varepsilon_1) = & 
-d^2 \hat{H}_0^0\bigr[d_2 \tilde{r}(v,\varepsilon_1), \mathsf{e}\bigr]
    -d^2 \hat{H}_0^0\bigr[v, d_2 d_1 \tilde{r}(v,\varepsilon_1)[\mathsf{e},1]\bigr]  \\-& d^2 \hat{H}^1_0\bigr[v,\mathsf{e}\bigr] -\frac{1}{\sigma_0-\frac{1}{3}}q
        + O(|(q,p,\varepsilon_1)|^2).
\end{aligned}
\end{equation}

Now it follows from Lemma~\ref{lemequ} and Lemma~\ref{Lem coff} that 
\[
d^2 \hat{H}_0^0\bigr[d_2 \tilde{r}(v,\varepsilon_1), \mathsf{e}\bigr] = 0, \qquad 
d^2 \hat{H}_0^0\bigr[v, d_2 d_1 \tilde{r}(v,\varepsilon_1)[\mathsf{e},1]\bigr] = O(|p|),
\]
and 
\[
d^2 \hat{H}^1_0\bigr[v,\mathsf{e}\bigr] = -\frac{1}{\sigma_0-\frac{1}{3}}q.
\]
This proves the estimate for $\partial_{\varepsilon_1}\mathcal R_2$, and hence the lemma.
\end{proof}

\begingroup
\setlength{\itemsep}{2pt}
\setlength{\parskip}{0pt}
\bibliographystyle{alpha}
\bibliography{Bibjournal.bib}

@misc{russell1844report,
  author       = {Russell, J. Scott},
  title        = {Report on Waves},
  howpublished = {Report of the 14th Meeting of the {British Association} for the {Advancement} of {Science}, John Murray, London},
  pages        = {311--390},
  year         = {1844}
}

@book{Schwartz1966,
  author    = {Schwartz, Laurent},
  title     = {Th{\'e}orie des distributions},
  series    = {Publications de l'Institut de Math{\'e}matique de l'Universit{\'e} de Strasbourg},
  volume    = {IX-X},
  publisher = {Hermann},
  address   = {Paris},
  year      = {1966},
  pages     = {xiii+420},
  note      = {Nouvelle {\'e}dition, enti{\`e}rement corrig{\'e}e, refondue et augment{\'e}e}
}

@book{Friedman1969,
  author    = {Friedman, Avner},
  title     = {Partial Differential Equations},
  publisher = {Holt, Rinehart and Winston, Inc.},
  address   = {New York--Montreal--London},
  year      = {1969},
  pages     = {vi+262}
}

@book{Abraham1988,
  author    = {Abraham, R. and Marsden, J. E. and Ratiu, T.},
  title     = {Manifolds, Tensor Analysis, and Applications},
  series    = {Applied Mathematical Sciences},
  volume    = {75},
  edition   = {2},
  publisher = {Springer},
  address   = {New York},
  year      = {1988},
  pages     = {x+654},
  isbn      = {0-387-96790-7},
  doi       = {10.1007/978-1-4612-1029-0}
}

@book{Conway1996,
  author    = {Conway, John B.},
  title     = {A Course in Functional Analysis},
  series    = {Graduate Texts in Mathematics},
  volume    = {96},
  edition   = {2},
  publisher = {Springer},
  address   = {New York},
  year      = {1990},
}

@book{mielke1991hamiltonian,
  author    = {Mielke, Alexander},
  title     = {{Hamiltonian} and {Lagrangian} Flows on Center Manifolds: With Applications to Elliptic Variational Problems},
  series    = {Lecture Notes in Mathematics},
  volume    = {1489},
  publisher = {Springer},
  address   = {Berlin},
  year      = {1991},
  doi       = {10.1007/BFb0097544}
}

@book{constantin2011nonlinear,
  author    = {Constantin, Adrian},
  title     = {Nonlinear Water Waves with Applications to Wave-Current Interactions and Tsunamis},
  series    = {CBMS-NSF Regional Conference Series in Applied Mathematics},
  volume    = {81},
  publisher = {SIAM},
  address   = {Philadelphia, PA},
  year      = {2011},
  doi       = {10.1137/1.9781611971873}
}

@book{kato2013perturbation,
  author    = {Kato, Tosio},
  title     = {Perturbation Theory for Linear Operators},
  series    = {Classics in Mathematics},
  volume    = {132},
  publisher = {Springer},
  address   = {Berlin},
  year      = {1995},
  edition   = {2},
  note      = {Reprint of the 1980 edition},
  doi       = {10.1007/978-3-642-66282-9}
}

@article{friedrichs1954existence,
  author  = {Friedrichs, Kurt O. and Hyers, D. H.},
  title   = {The existence of solitary waves},
  journal = {Communications on Pure and Applied Mathematics},
  volume  = {7},
  number  = {3},
  pages   = {517--550},
  year    = {1954},
  doi     = {10.1002/cpa.3160070305}
}

@article{thomas1977existence,
  author  = {Beale, J. Thomas},
  title   = {The existence of solitary water waves},
  journal = {Communications on Pure and Applied Mathematics},
  volume  = {30},
  number  = {4},
  pages   = {373--389},
  year    = {1977},
  doi     = {10.1002/cpa.3160300402}
}

@article{amick1981periodic,
  author  = {Amick, Charles J. and Toland, John Francis},
  title   = {On periodic water-waves and their convergence to solitary waves in the long-wave limit},
  journal = {Philosophical Transactions of the Royal Society of London. Series A, Mathematical and Physical Sciences},
  volume  = {303},
  number  = {1481},
  pages   = {633--669},
  year    = {1981}
}

@article{kirchgassner1988nonlinearly,
  author  = {Kirchg{\"a}ssner, Klaus},
  title   = {Nonlinearly resonant surface waves and homoclinic bifurcation},
  journal = {Advances in Applied Mechanics},
  volume  = {26},
  pages   = {135--181},
  year    = {1988},
  doi     = {10.1016/S0065-2156(08)70288-6}
}

@article{amick1989theory,
  author  = {Amick, Charles J. and Kirchg{\"a}ssner, Klaus},
  title   = {A theory of solitary water-waves in the presence of surface tension},
  journal = {Archive for Rational Mechanics and Analysis},
  volume  = {105},
  number  = {1},
  pages   = {1--49},
  year    = {1989},
  doi     = {10.1007/BF00251596}
}

@article{iooss1992water,
  author  = {Iooss, G{\'e}rard and Kirchg{\"a}ssner, Klaus},
  title   = {Water waves for small surface tension: an approach via normal form},
  journal = {Proceedings of the Royal Society of Edinburgh Section A: Mathematics},
  volume  = {122},
  number  = {3--4},
  pages   = {267--299},
  year    = {1992},
  doi     = {10.1017/S0308210500021119}
}

@article{Yi1993,
  author  = {Yi, Yingfei},
  title   = {Stability of integral manifold and orbital attraction of quasi-periodic motion},
  journal = {Journal of Differential Equations},
  volume  = {103},
  number  = {2},
  pages   = {278--322},
  year    = {1993},
  doi     = {10.1006/jdeq.1993.1051}
}

@article{buffoni1996plethora,
  author  = {Buffoni, B. and Groves, M. D. and Toland, John Francis},
  title   = {A plethora of solitary gravity-capillary water waves with nearly critical {Bond} and {Froude} numbers},
  journal = {Philosophical Transactions of the Royal Society of London. Series A: Mathematical, Physical and Engineering Sciences},
  volume  = {354},
  number  = {1707},
  pages   = {575--607},
  year    = {1996},
  doi     = {10.1098/rsta.1996.0020}
}

@article{iooss1996capillary,
  author  = {Iooss, G{\'e}rard and Kirrmann, Pius},
  title   = {Capillary gravity waves on the free surface of an inviscid fluid of infinite depth. {Existence} of solitary waves},
  journal = {Archive for Rational Mechanics and Analysis},
  volume  = {136},
  number  = {1},
  pages   = {1--19},
  year    = {1996},
  doi     = {10.1007/BF02199364}
}

@article{buffoni1999multiplicity,
  author  = {Buffoni, B. and Groves, M. D.},
  title   = {A multiplicity result for solitary gravity-capillary waves in deep water via critical-point theory},
  journal = {Archive for Rational Mechanics and Analysis},
  volume  = {146},
  number  = {3},
  pages   = {183--220},
  year    = {1999},
  doi     = {10.1007/s002050050141}
}

@article{groves2001spatial,
  author  = {Groves, M. D. and Mielke, A.},
  title   = {A spatial dynamics approach to three-dimensional gravity-capillary steady water waves},
  journal = {Proceedings of the Royal Society of Edinburgh Section A: Mathematics},
  volume  = {131},
  number  = {1},
  pages   = {83--136},
  year    = {2001},
  doi     = {10.1017/S0308210500000809}
}

@article{groves2002dimension,
  author  = {Groves, M. D. and Haragus, M. and Sun, S. M.},
  title   = {A dimension-breaking phenomenon in the theory of steady gravity-capillary water waves},
  journal = {Philosophical Transactions of the Royal Society of London. Series A: Mathematical, Physical and Engineering Sciences},
  volume  = {360},
  number  = {1799},
  pages   = {2189--2243},
  year    = {2002},
  doi     = {10.1098/rsta.2002.1066}
}

@article{haragus2002finite,
  author  = {Haragus, Mariana and Scheel, Arnd},
  title   = {Finite-wavelength stability of capillary-gravity solitary waves},
  journal = {Communications in Mathematical Physics},
  volume  = {225},
  number  = {3},
  pages   = {487--521},
  year    = {2002},
  doi     = {10.1007/s002200100590}
}

@article{groves2004steady,
  author  = {Groves, Mark D.},
  title   = {Steady water waves},
  journal = {Journal of Nonlinear Mathematical Physics},
  volume  = {11},
  number  = {4},
  pages   = {435--460},
  year    = {2004},
  doi     = {10.2991/jnmp.2004.11.4.2}
}

@article{buffoni2004existence,
  author  = {Buffoni, B.},
  title   = {Existence and conditional energetic stability of capillary-gravity solitary water waves by minimisation},
  journal = {Archive for Rational Mechanics and Analysis},
  volume  = {173},
  number  = {1},
  pages   = {25--68},
  year    = {2004},
  doi     = {10.1007/s00205-004-0310-0}
}

@article{buffoni2004existence2,
  author  = {Buffoni, B.},
  title   = {Existence by minimisation of solitary water waves on an ocean of infinite depth},
  journal = {Annales de l'Institut Henri Poincar{\'e} C, Analyse non lin{\'e}aire},
  volume  = {21},
  number  = {4},
  pages   = {503--516},
  year    = {2004},
  doi     = {10.1016/j.anihpc.2004.02.001}
}

@article{puaruau2005nonlinear,
  author  = {P{\u{a}}r{\u{a}}u, E. I. and Vanden-Broeck, J.-M. and Cooker, M. J.},
  title   = {Nonlinear three-dimensional gravity-capillary solitary waves},
  journal = {Journal of Fluid Mechanics},
  volume  = {536},
  pages   = {99--105},
  year    = {2005},
  doi     = {10.1017/S0022112005005136}
}

@article{wahlen2006steady,
  author  = {Wahl{\'e}n, Erik},
  title   = {Steady periodic capillary-gravity waves with vorticity},
  journal = {SIAM Journal on Mathematical Analysis},
  volume  = {38},
  number  = {3},
  pages   = {921--943},
  year    = {2006},
  doi     = {10.1137/050630465}
}

@article{Wahlen2007HamiltonianConstantVorticity,
  author  = {Wahl{\'e}n, Erik},
  title   = {A {Hamiltonian} formulation of water waves with constant vorticity},
  journal = {Letters in Mathematical Physics},
  volume  = {79},
  number  = {3},
  pages   = {303--315},
  year    = {2007},
  doi     = {10.1007/s11005-007-0143-5}
}

@article{groves2007spatial,
  author  = {Groves, Mark D. and Wahl{\'e}n, Erik},
  title   = {Spatial dynamics methods for solitary gravity-capillary water waves with an arbitrary distribution of vorticity},
  journal = {SIAM Journal on Mathematical Analysis},
  volume  = {39},
  number  = {3},
  pages   = {932--964},
  year    = {2007},
  doi     = {10.1137/070682283}
}

@article{groves2008fully,
  author  = {Groves, M. D. and Sun, S.-M.},
  title   = {Fully localised solitary-wave solutions of the three-dimensional gravity-capillary water-wave problem},
  journal = {Archive for Rational Mechanics and Analysis},
  volume  = {188},
  number  = {1},
  pages   = {1--91},
  year    = {2008},
  doi     = {10.1007/s00205-007-0085-1}
}

@article{deng2009three,
  author  = {Deng, Shengfu and Sun, Shu-Ming},
  title   = {Three-dimensional gravity-capillary waves on water---small surface tension case},
  journal = {Physica D: Nonlinear Phenomena},
  volume  = {238},
  number  = {17},
  pages   = {1735--1751},
  year    = {2009},
  doi     = {10.1016/j.physd.2009.05.011}
}

@article{constantin2011steady,
  author  = {Constantin, Adrian and Varvaruca, Eugen},
  title   = {Steady periodic water waves with constant vorticity: regularity and local bifurcation},
  journal = {Archive for Rational Mechanics and Analysis},
  volume  = {199},
  number  = {1},
  pages   = {33--67},
  year    = {2011},
}

@article{groves2011existence,
  author  = {Groves, M. D. and Wahl{\'e}n, Erik},
  title   = {On the existence and conditional energetic stability of solitary gravity-capillary surface waves on deep water},
  journal = {Journal of Mathematical Fluid Mechanics},
  volume  = {13},
  number  = {4},
  pages   = {593--627},
  year    = {2011},
  doi     = {10.1007/s00021-010-0034-x}
}

@article{martin2012regularity,
  author  = {Martin, Calin Iulian},
  title   = {Regularity of steady periodic capillary water waves with constant vorticity},
  journal = {Journal of Nonlinear Mathematical Physics},
  volume  = {19},
  number  = {suppl. 1},
  pages   = {1240006},
  year    = {2012},
  doi     = {10.1142/S1402925112400062}
}

@article{hur2012no,
  author  = {Hur, Vera Mikyoung},
  title   = {No solitary waves exist on 2{D} deep water},
  journal = {Nonlinearity},
  volume  = {25},
  number  = {12},
  pages   = {3301--3312},
  year    = {2012},
  doi     = {10.1088/0951-7715/25/12/3301}
}

@article{martin2013local,
  author  = {Martin, Calin Iulian},
  title   = {Local bifurcation for steady periodic capillary water waves with constant vorticity},
  journal = {Journal of Mathematical Fluid Mechanics},
  volume  = {15},
  number  = {1},
  pages   = {155--170},
  year    = {2013},
  doi     = {10.1007/s00021-012-0096-z}
}

@article{martin2013capillarygravity,
  author  = {Martin, Calin Iulian},
  title   = {Local bifurcation and regularity for steady periodic capillary-gravity water waves with constant vorticity},
  journal = {Nonlinear Analysis: Real World Applications},
  volume  = {14},
  number  = {1},
  pages   = {131--149},
  year    = {2013},
  doi     = {10.1016/j.nonrwa.2012.05.007}
}

@article{martinMatioc2013wilton,
  author  = {Martin, Calin Iulian and Matioc, Bogdan-Vasile},
  title   = {Existence of {Wilton} ripples for water waves with constant vorticity and capillary effects},
  journal = {SIAM Journal on Applied Mathematics},
  volume  = {73},
  number  = {4},
  pages   = {1582--1595},
  year    = {2013},
  doi     = {10.1137/120900290}
}

@article{grovesWahlen2015existence,
  author  = {Groves, Mark D. and Wahl{\'e}n, Erik},
  title   = {Existence and conditional energetic stability of solitary gravity-capillary water waves with constant vorticity},
  journal = {Proceedings of the Royal Society of Edinburgh Section A: Mathematics},
  volume  = {145},
  number  = {4},
  pages   = {791--883},
  year    = {2015},
  doi     = {10.1017/S0308210515000116}
}

@article{hsu2016gravity,
  author  = {Hsu, Hung-Chu and Francius, Marc and Montalvo, Pablo and Kharif, Christian},
  title   = {Gravity-capillary waves in finite depth on flows of constant vorticity},
  journal = {Proceedings of the Royal Society A: Mathematical, Physical and Engineering Sciences},
  volume  = {472},
  number  = {2195},
  pages   = {20160363},
  year    = {2016},
  doi     = {10.1098/rspa.2016.0363}
}

@article{kozlov2020solitary,
  author  = {Kozlov, Vladimir and Kuznetsov, Nikolai and Lokharu, Evgeniy},
  title   = {Solitary waves on constant vorticity flows with an interior stagnation point},
  journal = {Journal of Fluid Mechanics},
  volume  = {904},
  pages   = {A4},
  year    = {2020},
  doi     = {10.1017/jfm.2020.647}
}

@article{hur2020exact,
  author  = {Hur, Vera Mikyoung and Wheeler, Miles H.},
  title   = {Exact free surfaces in constant vorticity flows},
  journal = {Journal of Fluid Mechanics},
  volume  = {896},
  pages   = {R1},
  year    = {2020},
  doi     = {10.1017/jfm.2020.390}
}

@article{ifrim2020no,
  author  = {Ifrim, Mihaela and Tataru, Daniel},
  title   = {No solitary waves in 2{D} gravity and capillary waves in deep water},
  journal = {Nonlinearity},
  volume  = {33},
  number  = {10},
  pages   = {5457--5476},
  year    = {2020},
  doi     = {10.1088/1361-6544/ab999b}
}

@article{ifrim2022no,
  author  = {Ifrim, Mihaela and Pineau, Ben and Tataru, Daniel and Taylor, Mitchell},
  title   = {No pure capillary solitary waves exist in 2{D} finite depth},
  journal = {SIAM Journal on Mathematical Analysis},
  volume  = {54},
  number  = {4},
  pages   = {4452--4464},
  year    = {2022},
  doi     = {10.1137/22M1474929}
}

@article{kharif2025nonlinear,
  author  = {Kharif, Christian and Abid, Malek and Chen, Yang-Yih and Hsu, Hung-Chu},
  title   = {A nonlinear {Schr\"odinger} equation for capillary waves on arbitrary depth with constant vorticity},
  journal = {Journal of Fluid Mechanics},
  volume  = {1012},
  pages   = {A28},
  year    = {2025},
  doi     = {10.1017/jfm.2025.10324}
}

@article{barbieri2025bifurcation,
  author  = {Barbieri, Tommaso and Berti, Massimiliano and Maspero, Alberto and Mazzucchelli, Marco},
  title   = {Bifurcation of gravity-capillary {Stokes} waves with constant vorticity},
  journal = {Journal of Differential Equations},
  volume  = {451},
  pages   = {Paper No. 113753},
  year    = {2026},
  doi     = {10.1016/j.jde.2025.113753}
}

@article{hur2023unstable,
  author  = {Hur, Vera Mikyoung and Yang, Zhao},
  title   = {Stable and unstable capillary-gravity waves},
  journal = {Preprint, arXiv:2311.01368 [math.AP]},
  year    = {2023},
  url     = {https://arxiv.org/abs/2311.01368}
}

@article{rowan2024two,
  author  = {Rowan, James and Wan, Lizhe},
  title   = {Two-dimensional solitary water waves with constant vorticity, {Part II}: the deep capillary case},
  journal = {Preprint, arXiv:2408.03428 [math.AP]},
  year    = {2024},
  url     = {https://arxiv.org/abs/2408.03428}
}
\endgroup

\end{document}